\documentclass[reqno, a4paper]{amsart}
\usepackage{amssymb, mathdots}
\usepackage{mathabx}
\usepackage{mathrsfs}
\usepackage[utf8]{inputenc}
\usepackage{amsmath,  graphicx, tensor}
\usepackage{tikz}
\usetikzlibrary{matrix, arrows.meta}
\usetikzlibrary{cd}
\usepackage{enumerate}
\usepackage{hyperref}
\DeclareSymbolFont{bbold}{U}{bbold}{m}{n}
\DeclareSymbolFontAlphabet{\mathbbold}{bbold}
\def\qmod#1#2{{\hbox{}^{\displaystyle{#1}}}\!\big/\!\hbox{}_{
\displaystyle{#2}}}

 \def\psp#1#2%
  {\mathop{}%
   \mathopen{\vphantom{#2}}^{#1}%
   \kern-\scriptspace%
    \hskip -0.3mm{#2}} 
 \def\psb#1#2%
  {\mathop{}%
   \mathopen{\vphantom{#2}}_{#1}%
   \kern-\scriptspace%
    \hskip -0.0mm{#2}} 
 \def\pscr#1#2#3%
  {\mathop{}%
   \mathopen{\vphantom{#3}}^{#1}_{#2}%
   \kern-\scriptspace%
    \hskip -0.3mm{#3}} 
\def\C{{\mathbb C}}

\def\H{{\mathbb H}}

\def\N{{\mathbb N}}

\def\P{{\mathbb P}}

\def\R{{\mathbb R}}

\def\textmap#1{\mathop{\vbox{\ialign{
                                  ##\crcr
      ${\scriptstyle\hfil\;\;#1\;\;\hfil}$\crcr
      \noalign{\kern 1pt\nointerlineskip}
      \rightarrowfill\crcr}}\;}}
\def\bigtextmap#1{\mathop{\vbox{\ialign{
                                  ##\crcr
      ${\hfil\;\;#1\;\;\hfil}$\crcr
      \noalign{\kern 1pt\nointerlineskip}
      \rightarrowfill\crcr}}\;}}
      
\newcommand{\cal}{\mathcal}
\def\textlmap#1{\mathop{\vbox{\ialign{
                                  ##\crcr
      ${\scriptstyle\hfil\;\;#1\;\;\hfil}$\crcr
      \noalign{\kern-1pt\nointerlineskip}
      \leftarrowfill\crcr}}\;}}

\def\cg{{\mathfrak c}}
\def\dg{{\mathfrak d}}

\def\fg{{\mathfrak f}}
\def\g{{\mathfrak g}}
\def\hg{{\mathfrak h}}

\def\kg{{\mathfrak k}}
\def\lg{{\mathfrak l}}

\def\rg{{\mathfrak r}}

\def\Ng{{\mathfrak N}}

\theoremstyle{remark}
\newtheorem{ex}{Example}[section]

\newtheorem*{cl}{Claim}

\newtheorem{sz}{Satz}[section]
\theoremstyle{remark}
\newtheorem{re}[sz]{Remark} 
\theoremstyle{plain}
\newtheorem{thr}[sz]{Theorem}
\newtheorem{pr}[sz]{Proposition}
\newtheorem{co}[sz]{Corollary}
\newtheorem{dt}[sz]{Definition}
\newtheorem{lm}[sz]{Lemma}

\def\End{\mathrm {End}}
\def\Aut{\mathrm {Aut}}

\def\GL{\mathrm {GL}}
\def\gl{\mathrm {gl}}

\def\Hom{\mathrm{Hom}}

\def\id{ \mathrm{id}}
\def\im{\mathrm{im}}

\def\p{\mathrm{p}}

\newcommand\smvee{{\hskip -0.1ex \raise 0.2ex\hbox{$\scriptscriptstyle\vee$}}\hskip -0,3ex}

\newcommand{\extpw}{\mathchoice{{\textstyle\bigwedge}}%
    {{\bigwedge}}%
    {{\textstyle\wedge}}%
    {{\scriptstyle\wedge}}}
\def\extp{{\extpw}\hspace{-2pt}}

\def\Ad{\mathrm{Ad}}
\def\trp#1{\tensor[^{\mathrm{t}}]{#1}{}}
\def\edf{\coloneq}
\def\hb{\hbox}
\def\fr{\frac}
\def\p{\partial}
\def\bp{\bar\partial}

\begin{document}

\title{The Newlander-Nirenberg Theorem for principal bundles}

\author{Andrei Teleman}
 \thanks{The author would like to thank the anonymous referee for the careful reading of the  article, for his positive comments and useful remarks.}
\address{Aix Marseille Univ, CNRS, I2M, Marseille, France.}
\email[Andrei Teleman]{andrei.teleman@univ-amu.fr}
%\author{Matei Toma}
%\address{Université de Lorraine, CNRS, IECL, Nancy, France} 
%\email[Matei Toma]{matei.toma@univ-lorraine.fr}
\begin{abstract}
Let $G$ be an arbitrary (not necessarily isomorphic to a closed subgroup of $\mathrm{GL}(r,\mathbb{C})$) complex Lie group, $U$  a complex manifold and $p:P\to U$  a $\mathcal{C}^\infty$ principal $G$-bundle on $U$.  We introduce and study the space $\mathcal{J}^\kappa_P$ of  bundle almost complex structures of Hölder class $\mathcal{C}^\kappa$ on $P$. To any $J\in \mathcal{J}^\kappa_P$ we associate an $\mathrm{Ad}(P)$-valued form $\mathfrak{f}_J$ of type (0,2) on $U$ which should be interpreted as the obstruction to the integrability of $J$. For $\kappa\geq 1$ we have $\mathfrak{f}_J\in\mathcal{C}^{\kappa-1}(U,\bigwedge\hspace{-3.5pt}^{0,2}_{\,\,U}\otimes\mathrm{Ad}(P))$ whereas, for $\kappa\in[0,1)$, $\mathfrak{f}_J$ is a form with distribution coefficients.

Let $J\in \mathcal{J}^\kappa_P$ with $\kappa\in (0,+\infty]\setminus\mathbb{N}$.  We prove that $J$ admits locally $J$-pseudo-holomorphic sections of class $\mathcal{C}^{\kappa+1}$ if and only if $\mathfrak{f}_J=0$. If this is the case, $J$ defines a holomorphic reduction of the underlying $\mathcal{C}^{\kappa+1}$-bundle of $P$ in the sense of the theory of principal bundles on complex manifolds. The proof is based on classical regularity results for the $\bar\partial$-Neumann operator  on compact, strictly pseudo-convex complex manifolds with boundary.

The result will be used in  forthcoming articles dedicated to  moduli spaces of holomorphic bundles (on a compact complex manifold $X$) framed along a real hypersurface $S\subset  X$.

 \end{abstract}


\maketitle

\tableofcontents

{\ }\vspace{-10mm}\\
{\bf Notations:} We will use the notation $\extp^{d}_{\;U}$ ($\extp^{p,q}_{\;U}$) for the vector bundle of forms of degree $d$ (bidegree $(p,q)$) on a (complex) manifold $U$. The space of sections of class ${\cal C}^\infty$  in this bundle, i.e. the space of differrential forms  of degree $d$ (bidegree $(p,q)$) on $U$ will be denoted by $A^d(U)$ ($A^{p,q}(U)$). The notation $\Gamma(W,E)$ ($\Gamma(W,P)$) will stand for the set of ${\cal C}^\infty$  sections  of a differentiable vector  (principal) bundle above an open subset $W$ of its base manifold.

For a vector space (bundle) $E$ on a manifold (with boundary) $U$ ($\bar U$) we will use the notation ${\cal C}^\kappa(U, E)$ (${\cal C}^\kappa(\bar U, E)$) for the space of $E$-valued maps (sections in $E$) of class ${\cal C}^\kappa$ above $U$ ($\bar U$).  
\section{Introduction}

The classical Newlander-Nirenberg  theorem \cite{NN} states that  an almost complex structure $J$ on a differentiable manifold $M$ is integrable (i.e. is induced by a holomorphic structure on $M$) if and only if its Nijenhuis tensor $N_J$ vanishes. Several renowned proofs \cite{NN}, \cite{Ko}, \cite{FK}, \cite{We}, \cite{Ho2}, \cite{Ma}, \cite{HT} based on  different techniques   are available. 

The Newlander-Nirenberg theorem has a well-known version for vector bundles, which is due to Griffiths \cite[Proposition p. 419]{Gr} (see also \cite[Theorem 5.1]{AHS}, \cite[Proposition I.3.7]{Ko}, \cite[Theorem 2.1.53 and section 2.2.2]{DK}) and plays a crucial role in Gauge Theory (see for instance \cite{DK}). This vector bundle version of the Newlander-Nirenberg theorem states that a Dolbeault operator (semi-connection) $\delta:A^0(E)\to A^{0,1}(E)$ on a differentiable complex vector bundle $E$ on a complex manifold $U$ is integrable (i.e. is induced by a holomorphic structure on $E$) if and only if the form $F_\delta\in A^{0,2}(U,\End(E))$ associated with $\delta^2:A^0(E)\to A^{0,2}(E)$ vanishes (see section \ref{DolbeaultSect} in the Appendix).

The starting point of this article is  a natural generalization of this result to  principal $G$-bundles, where $G$ is an {\it arbitrary}  (not necessarily isomorphic to a subgroup of $\GL(r,\C)$) complex Lie group.

In the framework of principal bundles the role of Dolbeault operators is played by bundle almost complex structures, which are introduced and studied in section \ref{ACSsection}. A bundle almost complex structure (bundle ACS) on a principal $G$-bundle $p:P\to U$ is an almost complex structure $J$ on $P$ which makes the $G$-action on $P$ and the map  $p$  pseudo-holomorphic.  The space ${\cal J}_P$ of bundle ACS on $P$ is an affine space with model space $A^{0,1}(U,\Ad(P))$ and comes with a natural  action of the gauge group $\Aut(P)$ of $P$.

The principal bundle version of the Newlander-Nirenberg theorem states that (see section \ref{ACSsection} for details):
\begin{pr} \label{NeNi-smooth}
 A bundle ACS $J\in {\cal J}_P$ on $P$ is integrable if and only if the canonically associated form $\fg_J\in A^{0,2}(U,\Ad(P))$   vanishes.	\end{pr}
  The proof is an easy application of a well known remark related to the general Newlander-Nirenberg  theorem: the vanishing of the   Nijenhuis tensor $N_J$ of an ACS on $M$ is equivalent to the integrability of the   distribution $T^{0,1}_J\subset T_M^\C$ associated with $J$.

Proposition \ref{NeNi-smooth} is not sufficient for our purposes. In  \cite{TT}, \cite{Te2}, we  will construct and study moduli spaces of $S$-framed holomorphic bundles on a compact complex manifold $X$, where $S\subset X$ is a fixed real hypersurface. These moduli spaces are infinite dimensional. In order to endow such a moduli space with the structure of a Banach analytic  space in the sense of Douady, we have to work with bundle ACS belonging to a fixed Hölder differentiability class ${\cal C}^\kappa\edf {\cal C}^{[\kappa],\kappa-[\kappa]}$ with $\kappa\in (0,\infty)\setminus\N$ (see section \ref{H-spaces}) and to study the integrability condition in the appropriate sense for such a bundle ACS.

Our main result is a Hölder version of Proposition \ref{NeNi-smooth}: 
 \begin{thr}[Newlander-Nirenberg  theorem for bundle ACS of class ${\cal C}^\kappa$]  \label{NNGCkappa} 
 Let $G$ be a complex Lie group  and $p:P\to U$ a differentiable principal $G$-bundle on $U$. Let $J$ be a bundle ACS of class ${\cal C}^\kappa$ on $P$ with $\kappa\in(0,+\infty]\setminus\N$. The following conditions are equivalent:
 \begin{enumerate}
 	\item $\fg_J=0$.
 	\item   for any point $x\in U$ there exists an open neighborhood $W$ of $x$ and a $J$-pseudo-holomorphic section $\sigma\in {\cal C}^{\kappa+1}(W,P)$.
 \end{enumerate}
 \end{thr}

 Note that for $\kappa\in(0,1)$ the condition $\fg_J=0$ is meant in distributional sense, see Definition \ref{vanish-in-distribution-sense} in section \ref{proof-section}. Note also that the regularity class of $J$ in the general Newlander-Nirenberg  theorem for almost complex structures on manifolds of real dimension $\geq 4$ is ${\cal C}^\kappa$ with $\kappa>1$ \cite[Theorem 3.1]{We}, so the bundle  version of this theorem   requires a {\it weaker} regularity assumption.

In the theory of principal bundles on complex manifolds one does not use  or consider (almost) complex structures on the total spaces of the considered bundles. The fundamental objects in this theory are holomorphic  principal bundles in the following sense:
\begin{dt}\label{bundleCS}
 Let $P\stackrel{p}{\to} U$ be a topological principal $G$-bundle on $U$. A (bundle) holomorphic structure on $p$ is a set $\hg$ of continuous local sections $\tau:W_\tau\to P$ of $p$ such that:
 \begin{enumerate}
 	\item $\bigcup_{\tau\in \hg} W_\tau=U$.
 	\item Any two elements $\tau$, $\tau'\in \hg$ are holomorphically compatible, i.e. the comparison map $\psi_{\tau\tau'}: W_\tau\cap W_{\tau'}\to G$ defined by  $\tau'=\tau\psi_{\tau\tau'}$, is holomorphic.
 	\item $\hg$ is maximal (with respect to inclusion) satisfying (1), (2).
 \end{enumerate}
 A holomorphic principal $G$-bundle on $U$ is a pair $(P\stackrel{p}{\to} U,\hg)$ consisting of a topological principal $G$-bundle on $U$ and a (bundle) holomorphic structure on $p$.
 \end{dt}

The first consequence of  Theorem \ref{NNGCkappa} is
 \begin{co}\label{HolStr} Let $\kappa\in(0,+\infty]\setminus\N$.
 A bundle ACS $J$ of class ${\cal C}^\kappa$ with $\fg_J=0$ on $P$ defines a (bundle) holomorphic structure $\hg_J$ on its underlying topological bundle.  For an open set $W\subset U$, a section $\sigma\in {\cal C}^{1}(W,P)$ belongs to $\hg_J$ if and only if it is $J$-pseudo-holomorphic. 	
Moreover,   $\hg_J$ is contained   in the set of local sections of class ${\cal C}^{\kappa+1}$ of $P$, so it provides a holomorphic reduction of the underlying  ${\cal C}^{\kappa+1}$ bundle of $P$.
 
\end{co}

\begin{re}
Using the local trivializations associated with  sections $\tau\in \hg_J$, one obtains in particular a complex manifold structure  on $P$ which is compatible with $J$. Note that for $\kappa\in (0,1)$ one {\it cannot} obtain this structure using directly the general Newlander-Nirenberg theorem, because, as mentioned above,  in this theorem the required regularity class of $J$  is ${\cal C}^\kappa$ with $\kappa>1$.
\end{re}

\begin{co}
 \label{kappa-regularity}
 Let $U$ be a complex manifold, $G$ a complex Lie group, and $P$ a principal bundle of class ${\cal C}^\infty$ on $U$. Let $\kappa\in(0,+\infty]\setminus\N$, let $J$ be an   bundle	 ACS of class ${\cal C}^\kappa$  on $P$ with $\fg_J=0$, $G\times F\to F$  a holomorphic action of $G$ on a complex manifold $F$. The sheaf of   holomorphic (with respect to the holomorphic structure induced by $\hg_J$) local sections of the associated bundle $P\times_GF$  is contained in the sheaf of its local sections of class ${\cal C}^{\kappa+1}$.
\end{co}

Let $\theta \in A^1(G,\g)$ be the canonical left invariant $\g$-valued form on $G$ \cite[p. 41]{KN}. For a map $U\stackrel{\rm open}\supset V\textmap{\sigma}  G$ of class ${\cal C}^1$ put $\bar\lg(\sigma)\edf \sigma^*(\theta^{1,0})^{0,1}$. From an analytical point of view the meaning of Theorem \ref{NNGCkappa} is the following: for a $\g$-valued form $\alpha\in {\cal C}^\kappa (W,\extp^{0,1}_{\;W}\otimes \g)$ (with $\kappa\in(0,+\infty]\setminus\N$ and $W \subset U$ open), the non-linear first order differential equation
$$
\bar\lg(\sigma)=\alpha
$$
is locally solvable on $W$ if and only if $\bp\alpha+\frac{1}{2}[\alpha\wedge\alpha]=0$ in distributional sense.
\\

In the case $\dim(U)\geq 2$ the proof of Theorem \ref{NNGCkappa} is based on the following effective  version  of the Newlander-Nirenberg theorem for the trivial $G$-bundle on a strictly pseudo-convex  manifold: 
\begin{thr}\label{effective-NeNi}
Let  $X$ be a Hermitian manifold of dimension $n\geq 2$, and let $U\subset X$ be a   relatively compact strictly pseudoconvex open subset   with smooth boundary $\partial \bar U=\bar U\setminus U$. Suppose $H^q(U,{\cal O}_U)=0$ for $q\in\{1,2\}$ and let $\kappa\in(0,+\infty)\setminus\N$. There exists an open neighborhood $N_U$ of $0$ in the closed subset
$$W\edf \left\{\lambda\in {\cal C}^{\kappa}(\bar U,\extp^{0,1}_{\;\bar U}\otimes\g)|\   \bp \lambda+\frac{1}{2}[\lambda\wedge\lambda]=0\right\}\subset {\cal C}^{\kappa}(\bar U,\extp^{0,1}_{\;\bar U}\otimes\g),
$$
and, for any $\lambda\in N_U$,   a solution $u=\Ng(\lambda)\in {\cal C}^{\kappa+1}(U,\g)$ of the equation 
$$
\bar\lg(\exp(u))=\lambda
$$
which   satisfies estimates of the form
\begin{equation}
\|u|_V\|_{{\cal C}^{\kappa+1}}\leq C_V \|\lambda\|_{{\cal C}^\kappa}	
\end{equation}
for relatively compact subdomains $V\Subset U$. \end{thr}

The proof is based on classical regularity results for the $\bp$-Neumann operator \cite{FK}, \cite{LM}, \cite{BGS} and  elliptic interior estimates \cite{DN}.
\vspace{1mm}

In the case $\dim(U)= 1$ we have the following  effective  version  of the Newlander-Nirenberg theorem for the trivial $G$-bundle on a relatively compact open subset $U\subset\C$ with smooth boundary:

\begin{thr} \label{effective-NeNi-n=1}
Let $U\subset\C$ be a bounded domain   with smooth boundary and let $\kappa\in(0,+\infty)\setminus\N$. There exists an open neighborhood $N_U$	of 0 in ${\cal C}^\kappa(\bar U,\extp^{0,1}_{\;\bar U}\otimes\g)$ and, for any $\lambda\in N_U$,  a solution $u=\Ng_U(\lambda)\in   {\cal C}^{\kappa+1}(\bar U,\g)$ of the equation
$$
\bar\lg(\exp(u))=\lambda
$$
such that the obtained map $\Ng_U:N_U\to  {\cal C}^{\kappa+1}(\bar U,\g)$ is holomorphic and satisfies  $\Ng_U(0)=0$.
\end{thr}

The proof uses the ellipticity of the operator $\bp$ on the closed manifold $\P^1$ and a well known extension lemma for Hölder spaces.

\section{Bundle almost complex structures on principal bundles}\label{ACSsection}

\subsection{The Newlander-Nirenberg theorem for bundles in the smooth case}

Let $G$ be a complex Lie group and $\g$ its Lie algebra. Let $J_G\in\Gamma(G,\End_\R(T_G))$ be the almost complex structure on $G$ defining its complex structure and $J_\g\in\End_\R(\g)$ the endomorphism defining the complex structure of $\g$.   We obtain as usually direct sum decompositions
$$
T_G^\C=T_G^{1,0}\oplus T_G^{0,1},\ \g^\C=\g^{1,0}\oplus\g^{0,1}
$$
of the complexified tangent bundle, respectively Lie algebra of $G$.

Let $\theta\in A^1(G,\g)$ be the canonical left invariant form of $G$ \cite[p. 41]{KN}, and $\theta^{1,0}$ the composition
$$
T_G\otimes_\R\C\textmap{\theta\otimes_\R\id_\C} \g^\C\to \g^{1,0}.
$$
Since $\theta$ is holomorphic, $\theta\otimes_\R\id_\C$ preserves the type, so $\theta^{1,0}$ is a $\g^{1,0}$-valued form of type $(1,0)$; it can obviously be identified with $\theta$ via the standard isomorphisms $(T_G,J_G)\to  T_G^{1,0}$, $(\g,J_\g)\to \g^{1,0}$.\\

Let $p:P\to U$ be a differentiable principal $G$-bundle on $U$. Denote by $V\subset T_P$ the vertical distribution of $P$, and recall that this vector bundle comes with a   canonical trivialization $\vartheta:V\to P\times\g$ given by $(y,a)\mapsto a^\#_y$, which extends to a trivialization $\vartheta^\C:V^\C\to P\times\g^\C$ of the complexified vertical bundle. The complex structure $J_\g$ of $\g$ induces via $\vartheta$ a complex structure on the bundle $V$, so  a direct sum decomposition   $V^\C=V^{1,0}\oplus V^{0,1}$   which corresponds via $\vartheta^\C$ to the decomposition  $\g^\C=\g^{1,0}\oplus\g^{0,1}$.
The subbundle $p_*^{-1}(T^{0,1}_U)$ of $T_P^\C$ fits  in the short exact sequence
\begin{equation}\label{ShExSeqp*}
0 \to V^\C=V^{1,0}\oplus V^{0,1}\hookrightarrow p_*^{-1}(T^{0,1}_U)\textmap{p_*}p^*(T^{0,1}_U)\to 0.	
\end{equation}

\begin{dt}\label{BdACS}
A bundle almost complex structure  (ACS) on $P$ is an almost complex structure  $J$ on $P$ which makes the $G$-action $P\times G\to P$  and the map $p:P\to U$ pseudo-holomorphic. 
 \end{dt}
 Let $J$ be a bundle ACS on $P$. The subbundle $T^{0,1}_{P,J}\subset T_P^\C$ of type (0,1) tangent vectors with respect to $J$ is a $G$-invariant subbundle  of   $p_*^{-1}(T^{0,1}_U)$ which contains $V^{0,1}$ and is a complement of $V^{1,0}$ in $p_*^{-1}(T^{0,1}_U)$. Therefore one can write 
$$T^{0,1}_{P,J}=\ker(\alpha_J)$$
 for a well defined  section   $\alpha_J\in\Gamma(P,p_*^{-1}(T^{0,1}_U)^* \otimes\g^{1,0})$  with the following properties:
  \begin{enumerate}[(Pa)]
 	\item $\alpha_J$ is invariant with respect to the $G$ action $g\mapsto \trp{R}_{g*}\otimes \Ad_{g}$ on   $p_*^{-1}(T^{0,1}_U)^* \otimes\g$. 
 	\item  $\alpha_J$ agrees with the  $\g^{1,0}$-valued form  
 $$
 V^\C\textmap{\vartheta^\C} P\times (\g^{1,0}\oplus\g^{0,1})\to \g^{1,0}
 $$	
 	on $V^\C$. In other words  $\alpha_J$ vanishes on $V^{0,1}$ and induces the canonical isomorphism $V^{1,0}_y\textmap{\vartheta^\C_y} \g^{1,0}$ for any $y\in P$.
 \end{enumerate}        
 The subbundle $p_*^{-1}(T^{0,1}_U)\subset T^\C_P$ splits as a direct sum
 $$p_*^{-1}(T^{0,1}_U)=T^{0,1}_{P,J}\oplus V^{1,0},$$
  and the projection on the first summand is 
\begin{equation}\label{ProjOnT{0,1}}
 \beta_J:p_*^{-1}(T^{0,1}_U) \to T^{0,1}_{P,J}\,, \ \beta_J(v)= v-\alpha_J(v)^\#_y\  \hb{ for }	y\in P,\ v\in p_*^{-1}(T^{0,1}_U)_y.
\end{equation}
\begin{re}\label{bijection-J-alpha_J} Let $p:P\to U$ be a principal $G$-bundle on $U$.
\begin{enumerate}
	\item The assignment $J\mapsto \alpha_J$ gives a bijection between the set ${\cal J}_P$ of bundle ACS on $P$ and the set ${\cal A}_P$ of sections   
$\alpha\in\Gamma(P,p_*^{-1}(T^{0,1}_U)^* \otimes\g^{1,0})$
satisfying properties (Pa), (Pb).
	\item If $P=U\times G$ is the trivial bundle over $U$, the short exact sequence (\ref{ShExSeqp*}) comes with an obvious splitting, and ${\cal A}_P$ can be identified with $A^{0,1}(U,\g^{1,0})$. The product bundle ACS $J_0$ on $U\times G$ corresponds to $\alpha_{J_0}=0$.
	\item Let $J\in {\cal J}_P$. A local section $\tau\in \Gamma(W,P)$ of $P$ defines a trivialization $P_W\stackrel{\simeq}{\to} W\times\C$, so, by   (2), $\alpha_J$ gives a form $\alpha_J^\tau\in A^{0,1}(W,\g^{1,0})$. Explicitly, in terms of $\tau$, we have for any $v\in T_W$:
\begin{equation}
\alpha_J^\tau(v)\edf (\alpha_J\circ \tau_*)(v^{0,1})	.
\end{equation}
\item Let $J\in {\cal J}_P$. A local section $\tau\in \Gamma(W,P)$ of $P$ is $J$-pseudo-holomorphic if and only if 	$\alpha_J^\tau=0$.

\end{enumerate}

\end{re}

The map $\tau\mapsto \alpha_J^\tau$ satisfies the following transformation formula: 
\begin{re}\label{alpha-J-tau f-rem}
Let $f\in {\cal C}^\infty(W,G)$.  Then
\begin{equation}\label{alpha-J-tau f}
\alpha_J^{\tau f}=	\Ad_{f^{-1}}(\alpha_J^\tau)+f^*(\theta^{1,0})^{0,1}.
\end{equation}
\end{re}
\begin{proof}
Put $\tau'\edf \tau f$. For $y\in P$ denote by $l^y:G\to P$ the map $g\mapsto yg$.
For any $v\in T_{U,x}^\C$ we have
\begin{equation*}
\begin{split}
\tau'_*(v)=&R_{f(x)*}(\tau_*(v))+l^{\tau(x)}_*(f_*(v))=R_{f(x)*}(\tau_*(v))+l^{\tau(x)f(x)}_*(l_{f(x)*}^{-1}(f_*(v)))\\
=&%R_{f(x)*}(\tau_*(v))+\theta(f_*(v))^{\#}_{\tau(x)f(x)}=
R_{f(x)*}(\tau_*(v))+(\theta(f_*(v)))^{\#}_{\tau(x)f(x)}.
\end{split}
\end{equation*}
Using properties (Pa), (Pb) we obtain for any $v\in T^{0,1}_{U,x}$
$$
\alpha^{\tau'}_J(v)=\Ad_{f(x)^{-1}}(\alpha^\tau_J(v))+\theta^{1,0}(f_*(v)) =\Ad_{f(x)^{-1}}(\alpha^\tau_J(v))+f^*(\theta^{1,0})(v), 
$$
 which proves the claim.
 \end{proof}
 
 \begin{re}\label{bijection-for-J-hol-sections}
 Let $J\in {\cal J}_P$ and  $\tau\in \Gamma(W,P)$. The map 
 $\sigma\mapsto \tau\sigma^{-1}
 $
 induces a bijection between the set of solutions of the equation    \begin{equation}\label{J-pseudo-hol-eq}
 \alpha_J^\tau=\sigma^*(\theta^{1,0})^{0,1}
 \end{equation}
for $\sigma\in {\cal C}^\infty(W,G)$ and the  set of $J$-pseudo-holomorphic sections of $P$ on $W$. 
 \end{re}
 
 \begin{proof}
  Indeed, by Remark \ref{bijection-J-alpha_J} (4) we know that $\tau_\sigma\edf\tau\sigma^{-1}$  is $J$ pseudo-holomorphic if and only if $\alpha_J^{\tau_\sigma}=0$. Writing $\tau=\tau_\sigma \sigma$, formula (\ref{alpha-J-tau f}) shows that 	the equation $\alpha_J^{\tau_\sigma}=0$ is equivalent to (\ref{J-pseudo-hol-eq}).
 \end{proof}

 Let $J\in {\cal J}_P$, put $\alpha\edf \alpha_J$ and consider the anti-symmetric ${\cal C}^\infty(P,\C)$-bilinear map
$$
\Gamma(P,T^{0,1}_{P,J})^2\ni (A,B)\mapsto \alpha([A,B]). 
$$
Since the subbundle $T^{0,1}_{P,J}$ is $G$-invariant, it follows that $[a^\#,\cdot]$ leaves the space $\Gamma(P,T^{01}_{P,J})$ invariant for any $a\in \g^\C$, in particular $\alpha([A,B])=0$ if $A$ or $B$ is vertical. It follows that the formula
$$
\Gamma(P,T^{\C}_{P})^2\ni (A,B)\stackrel{\fg_J}{\longmapsto} -\alpha([A^{0,1}_J,B^{0,1}_J])
$$
defines a $\g^{1,0}$-valued tensorial  (0,2)-form of type $\Ad$ on $P$ (see \cite[section II.5]{KN}), i.e. an element of the space $A^{0,2}_\Ad(P,\g^{1,0})$. Identifying $\g^{1,0}$ with $\g$ in the canonical way, we may regard $\fg_J$ as a  $\g$-valued tensorial  form of type $(0,2)$ on $P$, i.e. as an element  of $A^{0,2}_\Ad(P,\g)=A^{0,2}(U,\Ad(P))$.  We will denote by the same symbol the corresponding element  of $A^{0,2}(U,\Ad(P))$.\\

With these notations we can prove the Newlander-Nirenberg theorem for principal bundles in the smooth case:
 
 \begin{proof} (of Proposition \ref{NeNi-smooth})
  The distribution $p_*^{-1}(T^{0,1}_X)\subset T^\C_P$ is obviously integrable (because it is the pull-back of $T^{0,1}_U$, and $U$ is a complex manifold) and contains $T^{0,1}_{P,J}$. 
For  vector fields $A$, $B\in \Gamma(P,T^{0,1}_{P,J})$ the Poisson bracket $[A,B]$ will still belong to $\Gamma(P,p_*^{-1}(T^{0,1}_U))$, but not necessarily to $\Gamma(P,T^{0,1}_{P,J})$; it belongs to this subspace if and only if $\alpha([A,B])=0$. Therefore the obstruction to the integrability of $J$ is the tensorial form $\fg_J$ as claimed.
	
 \end{proof}

Let $\tau\in\Gamma(W,P)$ be a smooth local section. Although  the pull back $\tau^*$ on forms is not necessarily type preserving, we have 
\begin{equation}\label{fg-J-in-A02}
\tau^*(\fg_J)\in A^{0,2}(W,\g^{1,0}).	
\end{equation}
Indeed, for a tangent vector $v\in T^{1,0}_{U,x}$, we have 
$$\tau_*(v)\in p_{*\tau(x)}^{-1}(T^{1,0}_{U,x})=T^{1,0}_{P,\tau(x)}\oplus V^{0,1}_{\tau(x)}.$$
Since $\fg_J(A,B)=0$ if $A$ or $B$ is of type $(1,0)$ or vertical, it follows that $\tau^*(\fg_J)(v,w)=0$ if $v$ or $w$ is of type (1,0). Taking into account this remark, we define, for a local section $\tau\in\Gamma(W,P)$: 
 \begin{equation}
 \fg^\tau_J\edf \tau^*(\fg_J)\in A^{0,2}(W,\g^{1,0})\simeq 	A^{0,2}(W,\g).
 \end{equation}

 \begin{pr}\label{fg-tau-J-prop}
 Let $\tau\in\Gamma(W,P)$ be a local section of $P$.	 Then
 \begin{equation}\label{fg-tau-J}
 \fg^\tau_J=\bp\alpha_J^\tau+\frac{1}{2}[\alpha_J^\tau\wedge\alpha_J^\tau].
 \end{equation}
 \end{pr}
 \begin{proof}
 Note first that $\tau$ defines a smooth map 
 $$
 \gamma^\tau:p^{-1}(W)\to G
 $$
 uniquely determined by the condition
$
 \tau(p(y))\gamma^\tau(y)=y$.
 Using this map we obtain a monomorphism 
 $$
 \tilde\tau: p^*(T^{0,1}_W)\to p_*^{-1}(T^{0,1}_W)
 $$
 of vector bundles on $p^{-1}(W)$ given by
 \begin{equation}
 \tilde\tau_y(v)=R_{\gamma^\tau(y)*}(\tau_*(v))	\ \forall y\in p^{-1}(W)\; \forall v\in T^{0,1}_{p(y)}.
 \end{equation}
By definition, $\tilde\tau$ verifies the $G$-invariance property:
\begin{equation}\label{G-inv-tilde-tau}
R_{g*}\circ \tilde\tau=	\tilde\tau\ \forall g\in G.
\end{equation}
Note that $\tilde\tau$ is just the ``horizontal lift" operator with respect to the unique flat connection on $P_W$ which makes $\tau$ parallel.

For a vector field $\xi\in A^{0,1}(W)$ let $\tilde\tau(\xi)$ be the section of $p_*^{-1}(T^{0,1}_W)\subset T^\C_{P_W}$ corresponding to $\xi$ via $\tilde\tau$. Formula (\ref{G-inv-tilde-tau}) shows that $\tilde\tau(\xi)$ is a $G$-invariant vector field on $P_W$. The obtained map 
$$\alpha_J(\tilde \tau(\xi)):P_W\edf p^{-1}(W)\to \g^{1,0}$$
will be $\Ad$-equivariant, so can be regarded as a section in the associated vector bundle $P_W\times_{\Ad}\g^{1,0}\subset P_W\times_{\Ad}\g^{\C}$.

 Let $\xi$, $\eta\in \Gamma(W,T^{0,1}_W)$. Using the notations introduced in section \ref{VectorFieldsOnP} of the Appendix, formula  (\ref{ProjOnT{0,1}}) shows that the projections of $\tilde\tau(\xi)$,  $\tilde\tau(\eta)$ on $T^{0,1}_{P,J}$ are given by
$$
 \tilde\tau(\xi)^{0,1}= \tilde\tau(\xi)-\alpha_J(\tilde \tau(\xi))^\nu, \  \tilde\tau(\eta)^{0,1}= \tilde\tau(\eta)-\alpha_J(\tilde \tau(\eta))^\nu,
$$
so, taking into account that $\tilde\tau$ commutes with $[\cdot,\cdot]$, and Remarks \ref{xi-nu-lambda-rem}, \ref{nu-cdot-lambda-nu-cdot-lambda'}:
\begin{equation}
\begin{split}
[\tilde\tau(\xi)^{0,1},\tilde\tau(\eta)^{0,1}]&=[ \tilde\tau(\xi),\tilde\tau(\eta)]+[\tilde\tau(\eta),\alpha_J(\tilde \tau(\xi))^\nu]-[\tilde\tau(\xi),\alpha_J(\tilde \tau(\eta))^\nu]\\
&\ \ \  +[\alpha_J(\tilde \tau(\xi))^\nu,\alpha_J(\tilde \tau(\eta))^\nu]=\\
&=[ \tilde\tau(\xi),\tilde\tau(\eta)]+\big(\tilde\tau(\eta)(\alpha_J(\tilde \tau(\xi)))\big)^\nu-\big(\tilde\tau(\xi)(\alpha_J(\tilde \tau(\eta)))\big)^\nu \\
&\ \ \ -[\alpha_J(\tilde \tau(\xi)),\alpha_J(\tilde \tau(\eta))]^\nu.
\end{split}
\end{equation}

Since $\tilde\tau(\eta)(\alpha_J(\tilde \tau(\xi))$, $\tilde\tau(\xi)(\alpha_J(\tilde \tau(\eta))$ and $[\alpha_J(\tilde \tau(\xi)),\alpha_J(\tilde \tau(\eta))]$ are $\g^{1,0}$-valued maps, property (Pb) gives
$$
\alpha_J\big(\big(\tilde\tau(\eta)(\alpha_J(\tilde \tau(\xi)))\big)^\nu\big)=\tilde\tau(\eta)(\alpha_J(\tilde \tau(\xi)), \ \alpha_J\big(\big(\tilde\tau(\xi)(\alpha_J(\tilde \tau(\eta)))\big)^\nu\big)= \tilde\tau(\xi)(\alpha_J(\tilde \tau(\eta)), 
$$
$$
\alpha_J\big([\alpha_J(\tilde \tau(\xi)),\alpha_J(\tilde \tau(\eta))]^\nu\big)=[\alpha_J(\tilde \tau(\xi)),\alpha_J(\tilde \tau(\eta))],
$$ 
so, taking into account the definition of $\fg_J$ and that $\tilde\tau$ commutes with $[\cdot,\cdot]$, 
\begin{equation*}
\begin{split}
\fg_J(\tilde\tau(\xi),\tilde\tau(\eta))=&-\alpha_J(\tilde\tau([\xi,\eta]))-\tilde\tau(\eta)(\alpha_J(\tilde \tau(\xi))+\tilde\tau(\xi)(\alpha_J(\tilde \tau(\eta))\\
&+[\alpha_J(\tilde \tau(\xi)),\alpha_J(\tilde \tau(\eta))].	
\end{split}	
\end{equation*}
Composing from the right with $\tau$ and taking into account that $\tilde\tau(\xi)$, $\tilde\tau(\eta)$ are tangent to $\im(\tau)$ and that their restrictions to $\im(\tau)$ coincide with $\tau_*(\xi)$, respectively $\tau_*(\eta)$, we obtain
\begin{align*}
\fg^\tau_J(\xi,\eta)&=\fg_J(\tau_*(\xi),\tau_*(\eta))=-\alpha_J^\tau([\xi,\eta])-\eta(\alpha_J^\tau(\xi))+\xi(\alpha_J^\tau(\eta))+[\alpha_J^\tau(\xi),\alpha_J^\tau(\eta)]\\
&=(d\alpha_J^\tau)(\xi,\eta)+\frac{1}{2}[\alpha_J^\tau\wedge\alpha_J^\tau](\xi,\eta)=(\bp\alpha_J^\tau)(\xi,\eta)+\frac{1}{2}[\alpha_J^\tau\wedge\alpha_J^\tau](\xi,\eta).
\end{align*}
For the last equality, we took into account that $\xi$, $\eta$ are vector fields of type $(0,1)$.
 \end{proof}
\subsection{The associated Dolbeault operator}\label{Dolbeault-J-rho}
	\label{DolbOp}
Let $\rho:G\to \GL(F)$ be a representation of $G$ on a finite dimensional complex vector space $F$ and $E^\rho\edf P\times_{\rho}F$ be the associated vector bundle. Let $J\in {\cal J}_P$ be a bundle ACS on $P$ and   $s\in A^0(U,E^\rho)$. 
Regard $s$ as an element of ${\cal C}^\infty_\rho(P,F)$, and note that the differential $\bp_J(s)\in A^{0,1}_J(P,F)$ is a tensorial $F$-valued (0,1)-form of type $\rho$ on $P$, so it can be regarded as an element $\bp^\rho_J s\in A^{0,1}(U,E^\rho)$. The obtained first order differential operator $\bp^\rho_J:A^0(U,E^\rho)\to A^{0,1}(U,E^\rho)$ is a Dolbeault operator on $E^\rho$. We will use the simpler notation $\bp_J=\bp^\rho_J$ when $\rho$ is obvious from the context.

Via the identification $\Gamma(W,E^\rho)\simeq {\cal C}^\infty(W,F)$ induced by a local section $\tau\in \Gamma(W,P)$, $\bp^\rho_J$  is given  by the formula
\begin{equation}\label{bp-J{rho,tau}}
\bp_J^{\rho,\tau}s=\bp s+  \rg(\alpha^\tau_J) s,
\end{equation}
where $\rg:\g\to \gl(F)$ is the Lie algebra morphism associated with $\rho$. The $\End(E^\rho)$-valued (0,2)-form  $F_{\bp^\rho_J}$ associated with $(\bp^\rho_J)^2:A^0(U,E^\rho)\to A^{0,2}(U,E^\rho)$ (which is the obstruction to the integrability of $\bp^\rho_J$) is given by
\begin{equation}\label{F-{bp-rho-J}}
F_{\bp^\rho_J}=\rg(\fg_J).
\end{equation}
Its pull back via $\tau$ is
$$F_{\bp^\rho_J}^\tau=\bp \rg(\alpha^\tau_J)+ \rg(\alpha^\tau_J)\wedge \rg(\alpha^\tau_J)\in A^{0,2}(W,\End(E^\rho)),$$
where $\wedge$ on the right  is induced by the wedge product of forms and   composition of endomorphisms. 

Formula (\ref{F-{bp-rho-J}}) shows  that the obtained map
$$
D_\rho:{\cal J}_P\to {\cal D}_{E^\rho}
$$
maps the space ${\cal J}_P^{\mathrm{int}}$ of integrable bundle ACS on $P$ into the space ${\cal D}^{\mathrm{int}}_{E^\rho}$ of integrable Dolbeault operators on $E^\rho$. 

\begin{re}
The map $D_{\rho_{\rm can}}$ associated with the canonical representation  
$$\rho_{\rm can}:\GL(r,\C)\to\GL(\C^r)$$
is a bijection and restricts to a bijection ${\cal J}^{\mathrm{int}}_P\to {\cal D}^{\mathrm{int}}_{E_P}$, where $E_P\edf P\times_{\GL(r,\C)}\C^r$. \end{re}

Note that the canonical form $\theta$ on $\GL(r,\C)$ can be written as $g^{-1}dg$, so, identifying $\gl(r,\C)^{1,0}$	 with $\gl(r,\C)$ in the standard way, the transformation formula (\ref{alpha-J-tau f}) becomes
$$
\alpha_J^{\tau f}=	\Ad_{f^{-1}}(\alpha_J^\tau)+f^{-1}\bp f,
$$
which is the well-known  transformation formula  for the $\gl(r,\C)$-valued (0,1) form associated with a Dolbeault operator in a trivialization.

 \subsection{The  affine space \texorpdfstring{${\cal J}_P$}{J} and its gauge symmetry}
 
 Taking into account the properties (Pa), (Pb) it follows that the space ${\cal A}_P$  has a natural structure of an affine space with model space $A^{0,1}_\Ad(P,\g^{1,0})\simeq A^{0,1}(U,\Ad(P))$ of tensorial  (0,1)-forms of type $\Ad$ with values in $\g^{1,0}\simeq\g $ on $P$. The space  ${\cal J}_P$ of bundle ACS on $P$ will also be regarded as an $A^{0,1}_\Ad(P,\g^{1,0})$-affine space via the bijection $J\mapsto\alpha_J$ given by  Remark \ref{bijection-J-alpha_J}.
 
 Let $\iota:G\to \Aut(G)$ be the group morphism which assigns to $g\in G$ the inner automorphism $\iota_g$. The group ${\cal C}^{\infty}_\iota(P,G)$ of $\iota$-equivariant maps $P\to G$ can be identified with the space of sections $\Gamma(U,\iota(P))$, where $\iota(P)\edf P\times_\iota G$ can be identified with the bundle of fiberwise automorphisms of $P$. Therefore the space ${\cal C}^{\infty}_\iota(P,G)$ can also be identified with the gauge group $\Aut(P)$. The gauge transformation $\tilde\sigma$ associated with $\sigma\in {\cal C}^{\infty}_\iota(P,G)$ is given explicitly by
\begin{equation}\label{tilde-sigma}
 \tilde\sigma(y)=y\sigma(y)\ \forall y\in P.	
\end{equation}

 When no confusion can occur, we will write $\sigma$ instead of $\tilde\sigma$ to save on notation.

 \begin{dt} \label{def-lg_J-kg-J} Let $J\in {\cal J}_P$.  We define
\begin{alignat*}{8}
\bar\lg_J&: {\cal C}^{\infty}_\iota(P,G)&\to\ & A^{0,1}_\Ad(P,\g^{1,0}), &&\bar\lg_J(\sigma)&=&\sigma^*(\theta^{1,0})^{0,1}_J &\\
 \bar \kg_J&:A^{0,1}_\Ad(P,\g^{1,0})&\ \to\ & A^{0,2}_\Ad(P,\g^{1,0}), &\ & \bar\kg_J(b)&=&\bp_J b+\frac{1}{2}[b\wedge b],	
\end{alignat*} 
where $\bp_J$ stands for $\bp_J^\Ad$ (see section \ref{Dolbeault-J-rho}) and  $A^{0,q}_\Ad(P,\g^{1,0})$ stands for the space of $\g^{1,0}$-valued tensorial forms of type $\Ad$ on $P$ of bidegree $(0,q)$.
\end{dt}
\begin{ex}\label{ex-lg-kg}
In the special case of the trivial bundle $U\times G$ endowed with the product bundle ACS $J_0$	 we obtain (via the  standard identification $\g\simeq\g^{1,0}$) the maps
\begin{alignat*}{8}
\bar\lg&: {\cal C}^{\infty}(U,G)&\to\ & {\cal C}^\infty(U,\extp^{0,1}_{\;U}\otimes\g), &&\bar\lg(\sigma)&=&\sigma^*(\theta^{1,0})^{0,1} &\\
 \bar \kg&:{\cal C}^\infty(U,\extp^{0,1}_{\;U}\otimes\g)&\ \to\ & {\cal C}^\infty(U,\extp^{0,2}_{\;U}\otimes\g), &\ & \bar\kg(b)&=&\bp b+\frac{1}{2}[b\wedge b].	
\end{alignat*} 
Note that formula (\ref{J-pseudo-hol-eq}) in Remark \ref{bijection-for-J-hol-sections} can be written $\alpha^\tau_J=\bar\lg(\sigma)$, whereas formula (\ref{fg-tau-J}) in Proposition \ref{fg-tau-J-prop} can be written $\fg_J^\tau=\bar\kg(\alpha^\tau_J)$.
\end{ex}

\begin{lm}\label{bar-lg-sigma}
Let  $\sigma\in {\cal C}^{\infty}_\iota(P_W,G)$ and $\tau\in \Gamma(W,P)$. Put $\sigma_\tau\edf \sigma\circ\tau\in {\cal C}^\infty(W,G)$. We have
$$
\tau^*(\bar \lg_J(\sigma))=\sigma_\tau^*(\theta^{1,0})^{0,1}+(\Ad_{\sigma_\tau^{-1}}-\id)(\alpha_J^\tau).  
$$
\end{lm}	
 
 \begin{proof} Let $v\in T_{U,x}^{0,1}$, $y=\tau(x)\in P$ and $w\edf \tau_*(v)\in (p_*)^{-1}(T^{0,1}_U)_y\subset T_{P,y}^\C$, $a\edf\alpha_J(w)\in\g^{1,0}$. Since $w^{0,1}_J=w-\alpha_J(w)^\#_y=w-a^\#_y$, we have:
\begin{equation}\label{tau*(lg(sigma))}
 \tau^*(\bar\lg_J(\sigma))(v)=\theta^{1,0}(\sigma_*(w^{0,1}_J))=\theta^{1,0}(\sigma_*(w-a^\#_y)).
\end{equation}
 In general, for any $c\in \g$  we have
 \begin{equation}\label{theta(sigma*(c))}
\begin{split}
\theta(\sigma_*(c^\#_y))&=\theta\big(\frac{d}{dt}|_0 \Ad_{e^{-t c}} (\sigma(y))\big)=\theta(-r_{\sigma(y)*}(c)+ l_{\sigma(y)*}(c))=\\
&=c-l_{\sigma(y)*}^{-1}\circ r_{\sigma(y)*}(c)=c-\Ad_{\sigma(y)^{-1}}(c).	
\end{split}
 \end{equation}
This implies
$$
(\theta\otimes\id_\C)(\sigma_*(c^\#_y))=(\id-\Ad_{\sigma(y)^{-1}})(c)\  \forall c\in\g^\C,
$$
in particular 
$$
\theta^{1,0}(\sigma_*(c^\#_y))=(\id-\Ad_{\sigma(y)^{-1}})(c)\  \forall c\in\g^{1,0},
$$
so (\ref{theta(sigma*(c))}) gives $
\theta^{1,0}(\sigma_*(a^\#_y))=(\id-\Ad_{\sigma(y)^{-1}})(a)$ and (\ref{tau*(lg(sigma))}) becomes
\begin{align*}
 \tau^*(\bar\lg_J(\sigma))(v)&=\theta^{1,0}(\sigma_*(\tau_*(v))-(\id-\Ad_{\sigma(y)^{-1}})(\alpha_J(\tau_*(v))\\
 &=\theta^{1,0}(\sigma_{\tau*}(v))+(\Ad_{\sigma(y)^{-1}}-\id )(\alpha_J^\tau(v)).
\end{align*}
This proves the claim.

 \end{proof}

 \begin{pr} \label{fg-{J+b}}
Let $J\in {\cal J}_P$ and $b\in A^{0,1}_\Ad(P,\g^{1,0})$. We have
 $$
 \fg_{J+b}=\fg_J+\bar\kg_J(b).
 $$	
\end{pr}
 \begin{proof}
 Put $J'\edf J+b$. Let $\tau\in\Gamma(W,P)$ be a local section of $P$, and note that the argument which justified formula (\ref{fg-J-in-A02}) gives $b^\tau\edf\tau^*(b)\in A^{0,1}(W,\g^{1,0})$. Taking into account formula (\ref{bp-J{rho,tau}}), it follows that the form  
 $$(\bp_J b)^\tau\edf \tau^*(\bp_J b)\in A^{0,2}(W,\g^{1,0}),$$
which corresponds to $\bp_J b$ in the local trivialization associated with $\tau$, is
$$
(\bp_J b)^\tau=\bp b^\tau+ [\alpha_{J}^\tau\wedge b^\tau].
$$
By Proposition \ref{fg-tau-J-prop}, we have
\begin{align*}\fg^\tau_{J'}&=\bp\alpha_{J'}^\tau+\frac{1}{2}[\alpha_{J'}^\tau\wedge\alpha_{J'}^\tau]=(\bp\alpha_{J}^\tau+\bp b^\tau)+\frac{1}{2}[(\alpha_{J}^\tau+b^\tau)\wedge(\alpha_{J}^\tau+b^\tau)]\\
 &=\fg^\tau_J+\bp b^\tau+ [\alpha_{J}^\tau\wedge b^\tau]+\frac{1}{2}[b^\tau\wedge b^\tau]=\fg^\tau_J+(\bp_J b)^\tau+\frac{1}{2}[b^\tau\wedge b^\tau],	
\end{align*}
 which obviously coincides with $\tau^*(\fg_J+\bar\kg_J(b))$. 
 \end{proof}

We let the group  ${\cal C}^\infty_{\iota}(P,G)$ act on ${\cal J}_P$ from the right by 
 $$
J\cdot \sigma=\tilde\sigma_{*}^{-1}\circ J\circ \tilde\sigma_{*}. 
 $$
 In other words $J\cdot \sigma$ is defined such that  the gauge transformation $\tilde\sigma$ associated with $\sigma$ becomes a pseudo-holomorphic map $(P,J\cdot \sigma)\to (P,J)$.
 Taking into account that $\sigma_*$ leaves the subbundle $p_*^{-1}(T^{0,1}_U)\subset T^\C_P$ invariant, it is easy to see that the corresponding $\Aut(P)$-action on ${\cal A}_P$ is:
 $$
 \alpha\cdot\sigma=\alpha \circ \tilde\sigma_*.
 $$
 \begin{pr} \label{J-cdot-sigma}
 Let $\sigma\in {\cal C}^\infty_{\iota}(P,G)$.  For any $J\in {\cal J}_P$ we have:
\begin{enumerate}
\item $J\cdot\sigma=J+\bar\lg_J(\sigma)$.
\item $\fg_{J\cdot\sigma}=\Ad_{\sigma^{-1}}(\fg_J)$.	
\end{enumerate}
 \end{pr}
 \begin{proof} Put $J'\edf J\cdot\sigma$.\\
(1) Let $\tau\in \Gamma(W,P)$ be a local section. As in Lemma \ref{bar-lg-sigma} put $\sigma_\tau\edf \sigma\circ\tau\in {\cal C}^\infty(W,G)$. We have
$$
\alpha_{J'}^\tau=\alpha_{J'}\circ \tau_*=\alpha_{J}\circ  \tilde \sigma_*\circ \tau_*=\alpha_J^{\tilde \sigma\circ \tau}
$$
on   $T^{0,1}_W$. By (\ref{tilde-sigma}) we know that
$
\tilde \sigma\circ \tau=\tau  \sigma_\tau$,
so, by the transformation formula (\ref{alpha-J-tau f}), it follows:
$$\alpha_J^{\tilde\sigma\circ\tau}=\Ad_{\sigma_\tau^{-1}}(\alpha_J^\tau)+\sigma_\tau^*(\theta^{1,0})^{0,1}.
$$
On the other hand, by Lemma \ref{bar-lg-sigma},  
$$
\tau^*(\bar\lg_J(\sigma))=\sigma_\tau^*(\theta^{1,0})^{0,1}+(\Ad_{\sigma_\tau^{-1}}-\id)(\alpha_J^\tau),
$$
which proves the claim.
\vspace{2mm}\\
(2)  For complex vector fields $\xi$, $\eta$ on $P$ we have:
\begin{align*}	
 \fg_{J'}(\xi,\eta)&=-\alpha_{J'}([\xi^{0,1}_{J'},\eta^{0,1}_{J'}])=-\alpha_J(\tilde\sigma_*([\xi^{0,1}_{J'},\eta^{0,1}_{J'}])=-\alpha_J([\tilde\sigma_*(\xi^{0,1}_{J'}),\tilde\sigma_*(\eta^{0,1}_{J'})])=\\
&=-\alpha_J([\tilde\sigma_*(\xi)^{0,1}_{J},\tilde\sigma_*(\eta)^{0,1}_{J}])=\fg_J(\tilde\sigma_*(\xi),\tilde\sigma_*(\eta)).
\end{align*}
For a tangent vector $v\in T_{P,y}^\C$ we have $\tilde\sigma_*(v)=R_{\sigma(y)*}(v)+ \theta(\sigma_*(v))^{\#}_{y}$, where the second term is vertical. Since $\fg_J$ is a tensorial 2-form, we obtain 
$$\fg_J(\tilde\sigma_*(\xi),\tilde\sigma_*(\eta))=\Ad_{\sigma^{-1}}(\fg_J)(\xi,\eta),$$
which proves the claim.
 \end{proof}
 
 Combining Proposition \ref{fg-{J+b}} with Proposition \ref{J-cdot-sigma}, we obtain: 
 $$
\Ad_{\sigma^{-1}}(\fg_J)= \fg_{J\cdot\sigma}=\fg_J+\bar\kg_J(\bar \lg_J(\sigma)), 
 $$
 so we obtain the following formula for the composition $\bar\kg_J\circ\bar\lg_J$.

 \begin{co}\label{kj-lj}
With the notations introduced in Definition \ref{def-lg_J-kg-J}, we have:
 $$
  \bar\kg_J\circ\bar\lg_J (\sigma)=(\Ad_{\sigma^{-1}}-\id)(\fg_J).
 $$
 In particular, if $\fg_J=0$, we have 	$ \bar\kg_J\circ\bar\lg_J=0$.
 \end{co}

\section{The  Hölder version   of the Newlander-Nirenberg  theorem for bundles}\label{proof-section}
Let $P$ be ${\cal C}^\infty$ principal  $G$-bundle on $U$ and $\kappa\in [0,+\infty]$.   A bundle ACS of class ${\cal C}^\kappa$  on $P$ is an almost complex structure on $P$ satisfying the conditions of Definition \ref{BdACS} which, regarded as section of the vector bundle $\End(T_P)$, is of class ${\cal C}^\kappa$. Equivalently, a bundle ACS of class ${\cal C}^\kappa$  on $P$ is a  bundle ACS of class ${\cal C}^0$    on $P$, such that, for any local section $\tau\in \Gamma(W,P)$, the form $\alpha^\tau_J\in A^{0,1}(W,\g^{1,0})$ is of class (has coefficients in) ${\cal C}^\kappa$.

We will denote by ${\cal J}^\kappa_P$ the space of bundle ACS of class ${\cal C}^\kappa$ on $P$. It is an affine space with model space the space ${\cal C}^\kappa(U,\extp^{0,1}_{\;U}\otimes \Ad(P))$ of $\Ad(P)$-valued $(0,1)$-forms of class ${\cal C}^\kappa$.
Throughout this section we will fix a Hermitian inner product on the Lie algebra $\g$.

\subsection{The case \texorpdfstring{$n=1$}{0}}

In the case $n=1$ Theorem \ref{NNGCkappa} states that, under the given assumptions, for {\it any}  $J\in{\cal J}^\kappa_P$ there exists a $J$-pseudo-holomorphic section of class ${\cal C}^{\kappa+1}$ around any point $x\in U$.

Consider first the case of the closed Riemann surface $\P^1$. Choosing a partition of unity subordinate to the standard atlas 
$$\big\{\P^1\setminus\{\infty\}\textmap{\simeq}\C,\   \P^1\setminus\{0\}\textmap{\simeq}\C\big\}$$
 of $\P^1$ we  obtain explicit norms on the Hölder spaces ${\cal C}^\kappa(\P^1,\g)$, ${\cal C}^\kappa(\P^1,\extp^{0,1}_{\;\P^1}\otimes\g)$, see section \ref{H-spaces}.

The kernel of the operator
$$ \bp:{\cal C}^{\kappa+1}(\P^1,\g)\to {\cal C}^\kappa(\P^1, \extp^{0,1}_{\;\P^1}\otimes\g) $$
 is the space of constant maps $\P^1\to \g$ (which will be denoted by $\g$ to save on notation). On the other hand, by Dolbeault theorem and  Hölder elliptic regularity, the cokernel of this operator is identified with $H^1(\P^1,{\cal O}_{\P^1}\otimes_\C\g)$, which vanishes, because 
$$H^1(\P^1,{\cal O}_{\P^1}\otimes_\C\g)=H^1(\P^1,{\cal O}_{\P^1})\otimes_\C\g=0.$$
Let $K$ be a closed complement of $\g$ in the Banach space ${\cal C}^{\kappa+1}(\P^1,\g)$. Such a complement exists by the Hahn-Banach theorem, because $\g$ is finite dimensional. It follows that the restriction $\bp_0\edf \bp|_K:K\to {\cal C}^\kappa(\P^1, \extp^{0,1}_{\;\P^1}\otimes\g)$ is an isomorphism of Banach spaces. We can now state

\begin{pr}\label{M-N}
There exists an open neighborhood $N$ of 0 in ${\cal C}^\kappa(\P^1, \extp^{0,1}_{\;\P^1}\otimes\g)$ and, for any $\lambda\in N$,  a solution $u=\Ng(\lambda)\in K\subset {\cal C}^{\kappa+1}(\P^1,\g)$ of the equation
$$
\bar\lg(\exp(u))=\lambda
$$
such that the obtained map $\Ng:N_U\to K$ is holomorphic and satisfies: $\Ng(0)=0$, $d\Ng(0)=\bp_0^{-1}$.
\end{pr}
\begin{proof}
We make use of Lemma \ref{holomorphy-lemma} proved below, taking in this Lemma $U=\P^1$, $V=\g$, $F=\g^{1,0}\simeq\g$, $\omega=\exp^*(\theta^{1,0})$.	 It follows that the composition $\bar\dg=\bar\lg\circ\exp$ defines a holomorphic map ${\cal C}^{\kappa+1}(\P^1,\g)\to {\cal C}^\kappa(\P^1, \extp^{0,1}_{\;\P^1}\otimes\g)$. Moreover, for $s=0$, the map $\omega_0\in {\cal C}^{\kappa+1}(\bar U,\Hom(\g,\g))$ is the constant map $\bar U\ni x\mapsto \id_\g$, so by Lemma \ref{holomorphy-lemma} (3), the differential at 0 of $\bar\dg$ is $\bp:{\cal C}^{\kappa+1}(\P^1,\g)\to {\cal C}^\kappa(\P^1, \extp^{0,1}_{\;\P^1}\otimes\g)$.  
 It follows that the differential of the restriction
$$
\bar\dg_K\edf\bar\dg|_K:K\to {\cal C}^\kappa(\P^1, \extp^{0,1}_{\;\P^1}\otimes\g)
$$
at 0 is the invertible operator $\bp_0$, so $\bar\dg_K$ is a local biholomorphism around 0. Let $M\subset K$, $N\subset {\cal C}^\kappa(\P^1, \extp^{0,1}_{\;\P^1}\otimes\g)$ be open neighborhoods of 0 in the respective spaces such that $\bar\dg_K$ induces a biholomorphism $\bar\dg_0:M\to N$. It suffices to put $\Ng\edf\bar\dg_0^{-1}$.
\end{proof}

Before stating Lemma \ref{holomorphy-lemma} used in the above proof, we need a brief preparation: Let  $V$, $F$ be  finite dimensional Hermitian vector spaces, and $\omega\in \Omega^{1,0}(V,F)$  a holomorphic $F$-valued form of bidegree $(1,0)$ on $V$, regarded as holomorphic map $V\to \Hom(V^{1,0}_\C,F)$, where $V_\C\edf V\otimes_\R\C$.  

Let $X$ be a complex manifold, $U\subset X$ be  relatively compact and open  such that $\bar U$ is either a submanifold with smooth boundary $\bar U\setminus U$, or $U=\bar U=X$, in which case $X$ being a closed complex manifold.  We choose a pair $({\cal A},(\chi_h)_{h\in {\cal A}})$ consisting of a finite atlas of $\bar U$ and a partition of unity subordinate to  the open cover $(V_h)_{h\in {\cal A}}$ to define effective  explicit norms  on the Hölder spaces on $\bar U$ (see section \ref{H-spaces}).

Let  $s \in {\cal C}^{\kappa+1}(\bar U,V)$.  The  form $\bp s\in {\cal C}^{\kappa}(\bar U,\extp_{\;\bar U}^{0,1}\otimes V)$ can be identified with  the ${\cal C}^{\kappa}(\bar U,\Hom_\C(T_U^{0,1},V_\C^{1,0}))$-component of the complexified differential 
$$ds \otimes\id_\C: U\to \Hom_\C(T_{U}\otimes\C,V_\C)=\Hom_\C(T_{U}^{1,0}\oplus T_{U}^{0,1},V_\C^{1,0}\oplus V_\C^{0,1}).$$
Put $\omega_s \edf \omega\circ s$, and note that
$$\omega_s  \in {\cal C}^{\kappa+1}(\bar U,\Hom(V^{1,0}_\C,F)).$$
This follows from Palais' composition principle    \cite[p. 39, second paragraph after Axiom (B§5)]{Pa}  applied to the map $\omega:V\to \Hom(V^{1,0}_\C,F)$    regarded as a differentiable fiber preserving map of trivial vector bundles 
$$\bar U\times V\to \bar U\times\Hom(V^{1,0}_\C,F))$$
 over $\bar U$.

\begin{lm}\label{holomorphy-lemma}
Under the assumptions  above we have:
\begin{enumerate}
	\item 

The form  $s ^*(\omega)^{0,1}$ is given by the section $\omega_s \cdot \bp s \in{\cal C}^\kappa(\bar U,\Hom_\C(T_{U}^{0,1},F))$, where $\cdot$ stands for the fiberwise bilinear vector bundle map
$$
\Hom_\C(V^{1,0}_\C,F)\times \Hom_\C(T_U^{0,1},V_\C^{1,0})\to \Hom_\C(T_U^{0,1},F)
$$
on $U$ given fiberwise by the compositions
$$
\Hom_\C(V^{1,0}_\C,F)\times\Hom_\C(T_{U,x}^{0,1},V_\C^{1,0})\to \Hom_\C(T_{U,x}^{0,1},F),\ x\in U.
$$
\item The map 	${\cal C}^{\kappa+1}(\bar U,V)\ni s \stackrel{L}{\mapsto} s ^*(\omega)^{0,1}\in {\cal C}^\kappa(\bar U,\extp^{0,1}_{\;\bar U}\otimes F)$ is holomorphic.
\item The differential of $L$ at a  \emph{holomorphic} element $s \in {\cal C}^{\kappa+1}(\bar U,V)$ is given by
\begin{equation}\label{DL(sigma)}
dL(s )(\dot s )=\omega_s \cdot\bp\dot s .
\end{equation}
\end{enumerate}
\end{lm}
\begin{proof}
The first claim follows by the definition of the pull back form 	$s ^*(\omega)$ and its $(0,1)$ component. 

The second claim follows using:
\begin{itemize}
	\item The already proved claim (1) which yields a continuous  bilinear map 
$$
 {\cal C}^{\kappa+1} (\bar U,\Hom(V_\C^{1,0},F))\times {\cal C}^{\kappa+1}(\bar U,\extp^{0,1}_{\;\bar U}\otimes V^{1,0})\to {\cal C}^{\kappa}(\bar U,\extp^{0,1}_{\;\bar U}\otimes F)
$$
of Banach spaces.
	\item The holomorphy of the map   $s\mapsto \omega_s$. This follows by Palais' differentiability theorem \cite[Theorem 11.3]{Pa} and a well known holomorphy criterion \cite[Theorem 13.16 p. 107] {Mu} in terms of $\C$-differentiability  for  maps between Banach spaces.
	\item  the fact that $s\mapsto\bp s$ induces a continuous $\C$-linear operator ${\cal C}^{\kappa+1}(\bar U,V)\to {\cal C}^{\kappa}(\bar U,\extp_{\;\bar U}^{0,1}\otimes V)={\cal C}^{\kappa}(\bar U,\Hom_\C(T_{\bar U}^{0,1},V_\C^{1,0}))$. 
\end{itemize}

The third claim follows using Leibniz rule applied to the continuous  bilinear map 
$$
 {\cal C}^{\kappa+1} (\bar U,\Hom(V_\C^{1,0},F)))\times {\cal C}^{\kappa+1}(\bar U,\extp^{0,1}_{\;\bar U}\otimes V^{1,0})\to {\cal C}^{\kappa}(\bar U,\extp^{0,1}_{\;\bar U}\otimes F)
$$
of Banach spaces mentioned above and noting that the term containing $\bp s$ vanishes if $s$ is holomorphic.
\end{proof}

We can prove now the effective Newlander-Nirenberg Theorem for principal bundles in the case $n=1$ we have stated in the introduction:  

\begin{proof} (of Theorem \ref{effective-NeNi-n=1})
The main ingredient in the proof is the existence of a continuous extension operator 
$${\cal E}^\kappa_{\bar U}:{\cal C}^\kappa(\bar U,\extp^{0,1}_{\;\bar U}\otimes\g)\to {\cal C}^\kappa(\P^1,\extp^{0,1}_{\;\P^1}\otimes\g).$$
This follows from the extension Lemma  \cite[Lemma 6.37]{GiTr} noting  that the form $d\bar z$ gives a trivialization of the line bundle $\extp^{0,1}_{\;\C}$ on $\C$, so an obvious isomorphism ${\cal C}^\kappa(\bar U,\extp^{0,1}_{\;\bar U}\otimes\g)\textmap{\simeq}{\cal C}^\kappa(\bar U,\g)$.

To complete the proof it suffices to put $N_U\edf ({\cal E}^\kappa_{\bar U})^{-1}(N)$, where  $N$ is the open neighborhood of 0 in ${\cal C}^\kappa(\P^1,\extp^{0,1}_{\;\P^1}\otimes\g)$ given by Proposition \ref{M-N}, and to define  
$$\Ng_U:N_U\to  {\cal C}^{\kappa+1}(\bar U,\g)$$
by $\Ng_U(\lambda)=\Ng({\cal E}^\kappa_{\bar U}(\lambda))|_{\bar U}$.

\end{proof}

\subsection{The case \texorpdfstring{$n\geq 2$}{1}}
\label{Neumann}

Suppose now $n\edf\dim(U)\geq 2$. Note first that if $J\in {\cal J}^\kappa_P$ with $\kappa\geq 1$,  then $\fg_J\in {\cal C}^{\kappa-1}(U,\extp^{0,2}_{\;U}\otimes \Ad(P))$ and the condition $\fg_J=0$ has an obvious sense. In fact this condition has sense even for $\kappa=0$:
\begin{dt}\label{vanish-in-distribution-sense}
Let $J\in {\cal J}^0_P$ be a continuous 	bundle ACS on $P$. We will say that $\fg_J$ vanishes in distributional sense, and we will write $\fg_J=0$, if for any local section $\tau\in \Gamma(W,P)$ the $(0,2)$-form $\fg^\tau_J=\bp\alpha_J^\tau+\frac{1}{2}[\alpha_J^\tau\wedge\alpha_J^\tau]$ vanishes in distributional sense, i.e. for any compactly supported form $\varphi\in A^{n,n-2}_c(W,\Ad(P)^*)$ we have
$$
\int_U \langle \alpha_J^\tau\wedge \bp \varphi\rangle =\frac{1}{2}\int_U\big\langle[\alpha_J^\tau\wedge\alpha_J^\tau]\wedge \varphi\big\rangle.
$$
\end{dt}
If $J\in {\cal J}^\kappa_P$ with $\kappa\geq 1$, this condition is equivalent to the vanishing of $\fg_J$ as element of ${\cal C}^{\kappa-1}(U,\extp^{0,2}_{\;U}\otimes \Ad(P))\subset{\cal C}^{0}(U,\extp^{0,2}_{\;U}\otimes \Ad(P))$.
\begin{re}\label{local-J-hol-implies-fJ=0}
Let $J\in {\cal J}^0_P$	be a continuous 	bundle ACS on $P$ and let $s:U\to P$ be a  $J$-pseudo-holomorphic section of class ${\cal C}^1$. Then $\fg_J$ vanishes   in distributional sense. 
\end{re}
\begin{proof}
Let $\tau:W\to P$ be a local section of class ${\cal C}^\infty$ and let $\sigma\in {\cal C}^1(W,G)$ be such that $s|_W=\tau \sigma^{-1}$. Since $J$ is of class ${\cal C}^0$ we know that   $\alpha^\tau_J\in {\cal C}^0(W,\g)$. By Remark \ref{bijection-for-J-hol-sections} we have  
$$
\bar\lg (\sigma)=\alpha^\tau_J.
$$

Let $(\sigma_n)$ be a sequence in ${\cal C}^\infty(W,G)$ converging in the ${\cal C}^1$ topology to $\sigma$. It follows that $(\bar\lg(\sigma_n))$ converges in the ${\cal C}^0$ topology to $\bar\lg(\sigma)=\alpha^\tau_J$.   By Corollary \ref{kj-lj} we have $\bar\kg(\bar\lg(\sigma_n))=0$, so
$$
\bp(\bar\lg(\sigma_n))+\frac{1}{2}[(\bar\lg(\sigma_n))\wedge(\bar\lg(\sigma_n))]=0
$$
for any $n\in\N$. Taking the limit for $n\to\infty$ in distributional sense, we obtain
$$
\bp \alpha^\tau_J+\frac{1}{2}[\alpha^\tau_J\wedge\alpha^\tau_J]=0,
$$
as claimed.
\end{proof}

Let $X$ be a Hermitian manifold of dimension $n\geq 2$, and let $U\subset X$ be a   relatively compact strictly pseudoconvex open subset   with smooth boundary $\partial \bar U=\bar U\setminus U$.    The $L^2$-structures  used in the arguments below  are associated with the Hermitian structure of $X$, whereas the Hölder spaces ${\cal C}^{\kappa}(\bar U,\extp^{0,q}_{\;\bar U}\otimes\g)$ are endowed with the explicit norms associated with a pair $({\cal A},(\chi_h)_{h\in {\cal A}})$ consisting of a finite atlas of $\bar U$ and a partition of unity subordinate to  the open cover $(V_h)_{h\in {\cal A}}$ (see section \ref{H-spaces} in the appendix).

Under these our strict pseudo-convexity assumption, the Dolbeault cohomology groups $H^q(U,{\cal O}_U)$ can be identified with the harmonic spaces $\H^{0,q}$ \cite[Theorem 4.1 p. 314]{LM}) for $q> 0$. If we assume that $X$ is Stein, these spaces vanish for all $q>0$ \cite[Theorem 7.9 p. 180]{LM}). \\

Let 
$$
P:L^2(  U,\g)\to L^2(U,\g)\cap {\cal O}(U,\g)
$$
be the Bergman projection on the space of $L^2$ holomorphic $\g$-valued functions  on $U$. Let $\kappa\in (0,+\infty)\setminus\N$ and let $k\edf[\kappa]$ be its integer part. Put  
$$
K \edf \{f\in \ker(P)\cap {\cal C}^1(U,\g)|\ \|\bp f\|_{{\cal C}^\kappa}<\infty\}, 
$$
and we endow this vector space with the norm
$$
\| f\|_K\edf \|\bp f\|_{{\cal C}^\kappa}.
$$
\begin{lm} \label{KZ01} Under the assumptions above, suppose $H^1(U,{\cal O}_U)=0$. Let $Z^{0,1}$ be the closed subspace of ${\cal C}^{\kappa}(\bar U,\extp^{0,1}_{\;\bar U}\otimes\g)$ defined by
$$Z^{0,1}\edf \{\lambda\in {\cal C}^{\kappa}(\bar U,\extp^{0,1}_{\;\bar U}\otimes\g)|\ \bp\lambda=0\},$$
where, for $k=0$, the condition  $\bp\lambda=0$ is meant in  distributional sense on $U$.  \begin{enumerate}
\item The operator $\bp$ induces a (norm preserving) isomorphism of normed spaces $\bp_0:K\textmap{\simeq} Z^{0,1}$,
in particular $K$ is a Banach space. 
\item $K$ is contained in ${\cal C}^{\kappa}(\bar U,\g)$ and the inclusion operator is continuous.
\item We have $K\subset {\cal C}^{\kappa+1}(U,\g)$. Moreover, for any relatively compact $V\Subset U$,  there exists $C_V>0$ such that for any $u\in K$ we have the estimate:
\begin{equation}
\|u|_{V}\|_{{\cal C}^{\kappa+1}(V)}\leq C_V\|u\|_K.	
\end{equation}

\end{enumerate}

\end{lm}
\begin{proof}
(1) It is clear that $\bp_0: K\to Z^{0,1}$ is injective and preserves the norm. For the surjectivity: Let $\lambda\in {\cal C}^{\kappa}(\bar U,\extp^{0,1}_{\;\bar U}\otimes\g)$ with $\bp\lambda=0$. 

Since the harmonic space $\H^{0,1}$ vanishes, the equation $\bp u=\lambda$ is solvable. More precisely, the corresponding canonical solution \cite[p. 209]{LM}, \cite[Corollary 3.2 p. 305]{LM}, \cite[Proposition 3.1.15]{FK}  $f=\bp^*N\lambda$  belongs to $\ker(P)$, so it belongs to $K$ because $\bp f=\lambda\in {\cal C}^{\kappa}(\bar U,\extp^{0,1}_{\;\bar U}\otimes\g)$. \vspace{2mm}\\
(2) The proof of (1) shows that the inverse of $\bp_0$ is the restriction of $\bp^*N$ to $Z^{0,1}$, so it suffices to show that $\bp^*N$ restricts to a continuous operator 
$${\cal C}^{\kappa}(\bar U,\extp^{0,1}_{\;\bar U}\otimes\g)\to {\cal C}^{\kappa}(\bar U,\g).$$
 By \cite[Theorem 1 (a)]{BGS} it follows that $N$ restricts to a continuous operator ${\cal C}^{\kappa}(\bar U,\extp^{0,1}_{\;\bar U}\otimes\g)\to {\cal C}^{\kappa+1}(\bar U,\extp^{0,1}_{\;\bar U}\otimes\g)$. Since $N$ takes values in $\mathrm{dom}(\Box)\subset \mathrm{dom}(\bp^*)$ \cite[p. 209]{LM}, on which $\bp^*$ is given by the first order differential operator $\vartheta$ \cite[p. 206]{LM}, it follows that $\bp^*N$ restricts to a continuous operator ${\cal C}^{\kappa}(\bar U,\extp^{0,1}_{\;\bar U}\otimes\g)\to {\cal C}^{\kappa}(\bar U,\g)$\footnote{The quoted theorem uses the "standard Lipschitz spaces" $\Lambda_\kappa$, where $\kappa>0$. For non-integer $\kappa$,  this space can be identified with the Hölder space ${\cal C}^{[\kappa],\kappa-[\kappa]}$ \cite[Propositions 6, 9 in section V.4 	and section VI.2.3]{St} which we denote ${\cal C}^\kappa$. Note also that in fact, by \cite[Theorem 2 (a)]{BGS},  $\bp^*N$ maps continuously $\Lambda_\kappa$ even to $\Lambda_{\kappa+\frac{1}{2}}$.
} as claimed. 
 \vspace{2mm}\\
(3) The first claim of (3) follows using standard regularity property of the first order elliptic operator $\bp+\bp^*:\bigoplus_{0\leq 2q\leq n} A^{0,2q}(U)\to \bigoplus_{1\leq 2q+1\leq n} A^{0,2q+1}(U)$. The second claim follows using interior estimates \cite[Theorem 4, p. 529]{DN} for the same operator taking into account (2) which gives an estimate of $\|f\|_{{\cal C}^0}$ in terms of $\|f\|_K=\|\bp f\|_{{\cal C}^\kappa}$.
\end{proof}
%
%\begin{re}\label{RealFunctionToBoundary}
%In the special case $X=\C^n$	 (when $U$ is just a strictly pseudoconvex domain with smooth boundary in $\C^n$) one can use \cite[Theorem 4, p. 529]{DN} to give an explicit formula for the constant $C_V$ in terms of explicitly in terms of the real function $\bar V\ni x\mapsto d(x,\p\bar U)$
%\end{re}
%
\begin{lm}\label{inclusion1}
The formula  
$\bar\dg(u)\edf \bar\lg(\exp(u))$
defines  a holomorphic  map 
$$\bar\dg:{\cal C}^{\kappa+1} (\bar U,\g )\to  {\cal C}^{\kappa}(\bar U,\extp^{0,1}_{\;\bar U}\otimes\g)$$
 whose image is contained in the closed subset
  $$
W\edf \left\{\lambda\in {\cal C}^{\kappa}(\bar U,\extp^{0,1}_{\;\bar U}\otimes\g)|\   \bp \lambda+\frac{1}{2}[\lambda\wedge\lambda]=0\right\}\subset {\cal C}^{\kappa}(\bar U,\extp^{0,1}_{\;\bar U}\otimes\g),
$$
 and whose  differential at 0 is $d(\bar \dg)(0)(\dot s)=\bp\dot s$.
\end{lm}
In the case $k=0$ (i.e. $0<\kappa<1$) the condition $\bp\lambda+\frac{1}{2}[\lambda\wedge\lambda]=0$ in the definition of $W$ is meant in  distributional sense on $U$.
\begin{proof}

We use Lemma \ref{holomorphy-lemma} taking $V=\g$, $F=\g^{1,0}\simeq\g$, $\omega=\exp^*(\theta^{1,0})$  regarded as holomorphic 1-form on $\g$. Noting that $\bar\dg(u)=(\exp\circ u)^*(\theta^{1,0})^{0,1}=u^*(\omega)^{0,1}$, we obtain $\bar\dg(u)\in {\cal C}^{\kappa}(\bar U,\extp^{0,1}_{\;\bar U}\otimes\g)$ by Lemma \ref{holomorphy-lemma} (1). The other claims follow by   Lemma \ref{holomorphy-lemma} (2), (3).
\end{proof}

Let ${\cal K}:{\cal C}^{\kappa}(\bar U,\extp^{0,1}_{\;\bar U}\otimes\g)\to {\cal C}^{\kappa}(\bar U,\extp^{0,1}_{\;\bar U}\otimes\g)$ be the  map defined by
$$
{\cal K}(\lambda)=\lambda+\frac{1}{2}(\bp^*N)[ \lambda\wedge\lambda].
$$
This map is well defined and holomorphic. Indeed, using the mentioned above regularity property of the operator $N$ and a standard multiplicative property of Hölder spaces, it follows that the second term of ${\cal K}$  is a continuous quadratic (2-homogeneous) map ${\cal C}^{\kappa}(\bar U,\extp^{0,1}_{\;\bar U}\otimes\g)\to {\cal C}^{\kappa}(\bar U,\extp^{0,1}_{\;\bar U}\otimes\g)$ \cite[section I.2]{Mu}. Therefore ${\cal K}$ is even polynomial in the sense of \cite[Definition I.2.8]{Mu}.
\begin{lm}\label{inclusion2}
Suppose $H^q(U,{\cal O}_U)=0$ for $q\in\{1,2\}$. Then ${\cal K}(W)\subset Z^{0,1}$.	
\end{lm}
\begin{proof}
Let $\lambda\in W$.	We have in distributional sense
$$
\bp{\cal K}(\lambda)=\bp \lambda +\frac{1}{2}\bp \bp^*N[\lambda\wedge\lambda]=-\frac{1}{2}[\lambda\wedge\lambda]+\frac{1}{2}\Box N [\lambda\wedge\lambda]-\frac{1}{2}\bp^*\bp N[\lambda\wedge\lambda]. 
$$
Since   the harmonic space $\H^{0,2}$ vanishes, we have $\Box N=\id$ on $L^2$ forms of type (0,2), so we get in distributional sense:
\begin{equation}\label{range}
\bp {\cal K}(\lambda)=-\frac{1}{2}\bp^*\bp N[\lambda\wedge\lambda].	
\end{equation}
The range of $\bp N$ is contained in the domain of $\bp^*$, because $N$ takes values in 
\begin{align*}
\mathrm{dom}(\Box)=\big\{f\in L^2(U, \extp^{0,2}_{\;U}\otimes\g)|\ &f\in \mathrm{dom}(\bp)\cap\mathrm{dom}(\bp^*),\\
& \bp f\in \mathrm{dom}(\bp^*),\ \bp^* f\in \mathrm{dom}(\bp)\big\}	
\end{align*}
(see \cite[p. 201]{LM}). Therefore the right hand term of (\ref{range})  belongs to $L^2$, more precisely it belongs to the range ${\cal R}(\bp^*)$ of $\bp^*$ as closed and densely defined operator on $L^2$  (see \cite[p. 185]{LM}).  But then (\ref{range}) shows that the distribution $\bp {\cal K}(\lambda)$ belongs to $L^2$, more precisely it belongs to the range ${\cal R}(\bp)$ of $\bp$ as closed and densely defined operator on $L^2$  (see \cite[Theorem 2.6 p. 187]{LM}). Since  ${\cal R}(\bp)\bot{\cal R}(\bp^*)$ (see \cite[Theorem 5.14, or Theorem 6.2]{LM}), we get $\bp {\cal K}(\lambda)=0$.
\end{proof}

For the differential $d{\cal K}(0)$ of ${\cal K}$ at 0  we have  $d{\cal K}(0)=\id$, so $d({\cal K}\circ \bar\dg)(0)=d\bar\dg(0)= \bp$ by Lemma \ref{inclusion1}. On the other hand, by Lemmas \ref{inclusion1}, \ref{inclusion2}, ${\cal K}\circ \bar\dg$ takes values in $Z^{0,1}$. Therefore
\begin{re}\label{composition}
The induced map $\cg\edf{\cal K}\circ\bar\dg|_K:K\to Z^{0,1}$ is also  holomorphic, and its differential at 0 is $d\cg(0)=\bp_0$, which is a (norm preserving) isomorphism of normed spaces by Lemma \ref{KZ01}. 	
\end{re}

\begin{pr}
\label{the-map-Ng}	
 Suppose $H^q(U,{\cal O}_U)=0$ for $q\in\{1,2\}$.
 There exists an open neighborhood $N_U$ of $0$ in $W$ and a continuous map $\Ng:N_U\to K$ such that $\Ng(0)=0$, $\bar\dg\circ \Ng=\id_{N_U}$ and
 $
 \lim_{\lambda\to 0}\frac{\|\Ng(\lambda)\|_{K}}{\|\lambda\|_{{\cal C}^{\kappa}}}=1$.
\end{pr}
\begin{proof}

By the local inverse theorem applied to ${\cal K}$ and $\cg$,  there exists:
\begin{itemize}
\item   an open neighborhood  $B$ of 0 in ${\cal C}^{\kappa}(\bar U,\extp^{0,1}_{\;\bar U}\otimes\g)$ on which ${\cal K}$ is injective,
  
\item  an open neighborhood $M\subset K$ of 0 in $K$ on which $\cg$ is injective, and such that  $R\edf \cg(M)$ is open in $Z^{0,1}$ and  $\cg$ induces a biholomorphism   $\cg_0:M\to R$.	 
\end{itemize}

We choose $M$  sufficiently small such that $\bar\dg(M)\subset B$. This is possible, because $\bar\dg$ is holomorphic, hence continuous.

The intersection ${\cal K}^{-1}(R)\cap W$ is open in $W$ because it coincides with the pre-image of $R$ via the restriction ${\cal K}|_W:W\to Z^{0,1}$ (see Lemma \ref{inclusion2}). It follows that $N_U\edf B\cap {\cal K}^{-1}(R)\cap W$ is an open neighborhood of $0$ in $W$. 

We claim that in fact $B\cap {\cal K}^{-1}(R)\subset W$, i.e. that $N_U=B\cap {\cal K}^{-1}(R)$. Indeed, for any $\lambda\in B\cap {\cal K}^{-1}(R)$ we have ${\cal K}(\lambda)\in R$, so $\cg_0^{-1}({\cal K}(\lambda))\in M\subset K$, so
\begin{equation}\label{{cal K}={cal K}}
{\cal K}(\bar\dg(\cg_0^{-1}({\cal K}(\lambda))))=({\cal K}\circ\bar\dg)|_K(\cg_0^{-1}({\cal K}(\lambda)))=\cg(\cg_0^{-1}({\cal K}(\lambda)))={\cal K}(\lambda).	
\end{equation}
But both $\lambda$ and $\bar\dg(\cg_0^{-1}({\cal K}(\lambda)))$ belong to $B$. The former because we have chosen $\lambda\in B\cap {\cal K}^{-1}(R)$, the latter because $\cg_0^{-1}({\cal K}(\lambda))\in M$ and we have chosen $M$ such that $\bar\dg(M)\subset B$. Therefore, since ${\cal K}$ is injective on $B$, formula (\ref{{cal K}={cal K}}) implies
\begin{equation}\label{lambda=bar-dg}
\lambda=\bar\dg(\cg_0^{-1}({\cal K}(\lambda)))	,
\end{equation}
in particular $\lambda\in W$ by Lemma \ref{inclusion1}, and the claim is proved.\vspace{1mm}

Put $\Ng\edf \cg_0^{-1}\circ {\cal K}|_{N_U}:N_U\to K$. Formula (\ref{lambda=bar-dg}) gives
\begin{equation}
\bar\dg \circ \Ng =\id_{N_U}.	
\end{equation}
On the other hand
$$
 \lim_{\lambda\to 0}\frac{\|\Ng(\lambda)\|_{K}}{\|\lambda\|_{{\cal C}^{\kappa}}}= \lim_{\lambda\to 0}\frac{\|\cg_0^{-1}({\cal K}(\lambda))\|_{K}}{\|{\cal K}(\lambda)\|_{{\cal C}^{\kappa}}}\frac{\|{\cal K}(\lambda)\|_{{\cal C}^{\kappa}}}{\|\lambda\|_{{\cal C}^{\kappa}}}=1,
 $$
 because the differentials $d(\cg_0^{-1})(0)$, $d({\cal K})(0)$ are isomorphisms of normed spaces.
\end{proof}

Theorem \ref{effective-NeNi} stated in the introduction follows from Proposition \ref{the-map-Ng}  taking into account Lemma \ref{KZ01} (3).  \vspace{2mm}

Now we can prove our Hölder version of the Newlander-Nirenberg theorem:  

\begin{proof} (of Theorem \ref{NNGCkappa}) Suppose first $\kappa\in(0,+\infty)\setminus\N$.\\
 (2)$\Rightarrow$(1)  follows from Remark \ref{local-J-hol-implies-fJ=0}.
 \vspace{2mm}\\
 (1)$\Rightarrow$(2): Let $J\in {\cal J}^\kappa_P$ such that, in the case $n\geq 2$, we have  $\fg_J=0$. The problem is local, so we can assume that
\begin{itemize}
\item $U$ is an open neighborhood of 0  in   $\C^n$ and $x=0$.
\item $P$ is the trivial $G$-bundle $U\times G$ on $U$.
\end{itemize}

Let $\alpha\in {\cal C}^\kappa(U,\extp^{0,1}_{\;U}\otimes\g)$ be the form which corresponds to $J$ via the identifications explained in Remark \ref{bijection-J-alpha_J} (1), (2). Note first that, by Proposition \ref{fg-tau-J-prop}, the assumption $\fg_J=0$ (in the case $n\geq 2$) becomes
\begin{equation}\label{alpha-in-W}
\bp\alpha+\frac{1}{2}[\alpha\wedge\alpha]=0
\end{equation}
 (in distributional sense for $\kappa\in(0,1)$). Let $r>0$ be sufficiently small such that $\bar B_r\subset U$, where $B_r$ stands for the radius $r$ ball  around 0.  Taking into account Remark \ref{bijection-for-J-hol-sections} (generalized in the obvious way for bundle ACS of class ${\cal C}^\kappa$) it suffices to prove: 
\begin{cl}
For sufficiently small $\varepsilon\in(0,1]$ the equation 
$$
\bar\lg(\exp(u))=\alpha|_{B_{\varepsilon r}}
$$
has a solution $u\in {\cal C}^{\kappa+1}(B_{\varepsilon r},\g)$.	
\end{cl}
To prove this claim note that, since $U$ is an open subset of $\C^n$, $\alpha$ is given by a map 
$$\tilde\alpha:U\to  \Hom_\R(\C^n,\g)$$
 of class ${\cal C}^\kappa$ taking values in the space of anti-linear maps $\C^n\to\g$. Moreover, we have $\|\alpha|_{\bar B_{\varepsilon r}} \|_{{\cal C}^\kappa}=\|\tilde \alpha|_{\bar B_{\varepsilon r}} \|_{{\cal C}^\kappa}$  (see section \ref{H-spaces}). Let $h_\varepsilon:\bar B_r\to \bar B_{\varepsilon r}$ be the contraction $h_\varepsilon(z)=\varepsilon z$. Put
 $$
 \alpha_\varepsilon\edf h_\varepsilon^*(\alpha|_{\bar B_{\varepsilon r}})\in {\cal C}^\kappa(\bar B_r,\extp^{0,1}_{\;\bar B_r}\otimes\g).
 $$ 
 The corresponding map $\tilde\alpha_\varepsilon:\bar B_r\to \Hom_\R(\C^n,\g)$ is $ \tilde\alpha_\varepsilon=\varepsilon\tilde\alpha\circ h_\varepsilon$.
 This shows that, denoting as usual $k\edf[\kappa]$, $\nu\edf\kappa-k$,  
 \begin{itemize}
 \item for any multi-index $\beta\in \N^{2n}$ with $|\beta|\leq k$ and for any $x\in \bar B_r$ we have
 $$
 \p^\beta \tilde \alpha_\varepsilon(x)=\varepsilon^{|\beta|+1} (\p^\beta\tilde\alpha)(\varepsilon x).
 $$	
 \item for any multi-index $\beta\in \N^{2n}$ with $|\beta|= k$ and for any $x$, $y\in \bar B_r$ we have
 $$
  \frac{\|\p^\beta \tilde  \alpha_\varepsilon(x)-\p^\beta \tilde \alpha_\varepsilon(y)\|}{\|x-y\|^{\nu}}=\varepsilon^{\kappa+1} \frac{\|\p^\beta \tilde \alpha(\varepsilon x)-\p^\beta \tilde \alpha(\varepsilon y)\|}{\|\varepsilon x-\varepsilon y\|^{\nu}}.
 $$	
 \end{itemize}
 Therefore for any $\varepsilon\in (0,1]$ we have
\begin{equation}\label{epsilon-small}
 \|\alpha_\varepsilon\|_{{\cal C}^\kappa}=\|\tilde\alpha_\varepsilon\|_{{\cal C}^\kappa}\leq \varepsilon \|\tilde \alpha|_{\bar B_{\varepsilon r}}\|_{{\cal C}^\kappa}=\varepsilon  \|\alpha|_{\bar B_{\varepsilon r}}\|_{{\cal C}^\kappa}\leq \varepsilon \|\alpha|_{\bar B_r}\|_{{\cal C}^\kappa}.
\end{equation}
\begin{itemize} 
\item[-]  Suppose $n=1$. We apply Theorem 	\ref{effective-NeNi-n=1} to the bounded domain $B_r\subset\C$. Formula (\ref{epsilon-small}) shows that, for sufficiently small $\varepsilon>0$ we have $\alpha_\varepsilon\in N_{B_r}$, so the equation $\bar\lg(\exp(u))=\alpha_\varepsilon $ has a solution $u_\varepsilon\in {\cal C}^{\kappa+1}(\bar B_r,\g)$. Therefore $u_\varepsilon\circ h_\varepsilon^{-1}\in {\cal C}^{\kappa+1}(\bar B_{\varepsilon r},\g)$ is a solution of the equation $\bar\lg(\exp(u))=\alpha|_{\bar B_{\varepsilon r}}$. \vspace{2mm}
\item[-]  Suppose $n\geq 2$. We apply Theorem \ref{effective-NeNi} to the strictly pseudo-convex open subset $B_r$ of $X=\C^n$. By formula (\ref{alpha-in-W}) we have $\alpha_\varepsilon\in W$ for any $\varepsilon\in(0,1]$. Moreover, formula (\ref{epsilon-small}) shows that, for sufficiently small $\varepsilon>0$ the form $\alpha_\varepsilon$ belongs to the open neighborhood $N_{B_r}$ of 0 in $W$ given by Theorem \ref{effective-NeNi}, so the equation $\bar\lg(\exp(u))=\alpha_\varepsilon $ has a solution $u_\varepsilon\in {\cal C}^{\kappa+1}(B_r,\g)$. Therefore $u_\varepsilon\circ h_\varepsilon^{-1}\in {\cal C}^{\kappa+1}( B_{\varepsilon r},\g)$ is a solution of the equation $\bar\lg(\exp(u))=\alpha|_{B_{\varepsilon r}}$.

\end{itemize}

For $\kappa=+\infty$ the claim follows from Proposition \ref{NeNi-smooth}: in this case $J$ is an integrable bundle ACS of class ${\cal C}^\infty$ on $P$ and the bundle map $p:P\to U$ becomes a holomorphic submersion.  Local holomorphic sections of $p$ will be of class ${\cal C}^\infty$.
 \end{proof}
 \begin{re}
 Let $\kappa\in (0,+\infty]\setminus\N$. In  the case when $G$ is a complex Lie subgroup of $\GL(r,\C)$, the equation $\bar\lg(\sigma)=\alpha$ can be written as $\sigma^{-1}\bp\sigma=\alpha$. One can then use elliptic regularity and bootstrapping to prove that, for $\alpha\in {\cal C}^\kappa(W,\g)$	any solution in ${\cal C}^1(W,G)$ of the equation $\bar\lg(\sigma)=\alpha$ belongs to ${\cal C}^{\kappa+1}(W,G)$. Therefore, for a bundle ACS $J\in {\cal J}^\kappa_P$, any local $J$-pseudo-holomorphic section of class ${\cal C}^1$ is of class ${\cal C}^{\kappa+1}$.
 \end{re}

We can prove now Corollaries \ref{HolStr}, \ref{kappa-regularity} stated in the introduction:

\begin{proof} (of Corollary \ref{HolStr})

Let $\hg_J$ be the set of $J$-pseudo-holomorphic local sections of $P$ which are of class  ${\cal C}^{\kappa+1}$. It suffices to prove that $\hg_J$ is a holomorphic structure on (the underlying topological bundle of) $P$ in the sense of Definition \ref{bundleCS}. By Theorem \ref{NNGCkappa} $\hg_J$ satisfies condition (1) in this definition. 
In order to prove the second condition   (holomorphic compatibility), let  $\tau:W\to P$, $\tau':W'\to P$ be $J$-pseudo-holomorphic local sections of class ${\cal C}^{\kappa+1}$ of $P$. We have to prove that the  comparison map $\psi_{\tau\tau'}:W\cap W'\to G$ is holomorphic. 
The maps $\Psi: W\times G\to P_W$, $\Psi': W'\times G\to P_{W'}$ defined by
$$
\Psi(x,g)=\tau(x)g, \ \Psi'(x,g)=\tau'(x)g
$$
are  diffeomorphisms of class ${\cal C}^{\kappa +1}$. Moreover, they are $J$-pseudo-holomorphic because  $\tau$, $\tau'$ are $J$-pseudo-holomorphic and the action $P\times G\to P$ is $J$-pseudo-holomorphic. It follows that the ${\cal C}^{\kappa +1}$ diffeomorphism     
$$ \Psi^{-1}\circ\Psi':(W\cap W')\times G\to (W\cap W')\times G$$
is holomorphic. But 
$$
\Psi^{-1}\circ\Psi'(x,g)=(x, \psi_{\tau\tau'}(x)g),
$$
in particular $\Psi^{-1}\circ\Psi'(x,e)=(x, \psi_{\tau\tau'}(x))$, which proves that $\psi_{\tau\tau'}$ is holomorphic.

Finally note that $\hg_J$  is maximal (with respect to inclusion) satisfying (1), (2). Indeed,  a  local continuous section $\sigma$ of $P$ which is  holomorphically compatible with any $\tau\in\hg_J$ is obviously $J$-pseudo-holomorphic and of class ${\cal C}^{\kappa+1}$, so it belongs to $\hg_J$.
%
%because
%%
%$$
%\Psi (x,g\psi_{\tau\tau'}(x))=\tau(x) \psi_{\tau\tau'}(x)g=\tau'(x)g=\Psi'(x,g)
%$$
%$\tau'=\tau\psi_{\tau\tau'}$
%	
\end{proof}
\begin{proof} (of Corollary \ref{kappa-regularity})

Put $E\edf P\times_G F$, and let $\varphi:V\to E$ be a holomorphic (with respect to $\hg_J$) local section. This means that the corresponding $G$-equivariant map $\hat \varphi:P_V\to F$ is holomorphic with respect to $\hg_J$. We have to prove that $\varphi$ is of class ${\cal C}^{\kappa+1}$, i.e. that the composition  $\hat \varphi\circ \sigma:  W_\sigma \to F$ is of class ${\cal C}^{\kappa+1}$ for any ${\cal C}^\infty$ local section $\sigma:W_\sigma\to P$ with $W_\sigma\subset V$ . Let $x\in  W_\sigma$ and let $\tau:W_\tau\to P$ be a local section belonging to $\hg_J$ with $x\in W_\tau\subset   W_\sigma$. Since we assumed that $\varphi$ is holomorphic, we know that $\hat\varphi\circ \tau:W_\tau\to F$ is holomorphic. For $y\in W_\tau$ we have 
$$(\hat \varphi\circ \sigma)(y)=\hat \varphi(\sigma(y))=\hat \varphi(\tau(y)\psi_{\tau\sigma}(y))=\psi_{\tau\sigma}(y)^{-1}(\hat\varphi\circ \tau)(y),$$
so $\hat \varphi\circ \sigma$ is of class ${\cal C}^{\kappa+1}$ on $W_\tau$ because $\hat\varphi\circ \tau$ is holomorphic and, since $\tau$ is of class ${\cal C}^{\kappa+1}$,  the comparison map $\psi_{\tau\sigma}:W_\tau\to G$ is of class ${\cal C}^{\kappa+1}$.
\end{proof}

 \section{Appendix}
 
 \subsection{Lipschitz spaces, Hölder spaces}
\label{H-spaces} 
Let $\kappa\in (0,+\infty)\setminus\N$, $k\edf[\kappa]$, $\nu\edf\kappa-k$.  For a finite dimensional normed space  $T$   let $\mathrm{Lip}^\kappa(\R^n,T)$ be the order $\kappa$ Lipschitz space of $T$-valued maps on $\R^n$ in supremum norm   \cite[p. 2]{JW},  \cite[p. 176]{St}:
 \begin{equation}\label{defLipkappa}
 \mathrm{Lip}^\kappa(\R^n,T) \edf \{f\in {\cal C}^{[\kappa]}(\R^n,T)|\ \| f\|_{\mathrm{Lip}^\kappa}<\infty\},
 \end{equation}
 where
\begin{equation}
\begin{split}\label{DefLipkappa-norm}
\| f\|_{\mathrm{Lip}^\kappa} \edf    \inf & \big\{m\in\R_+|\ \sup_{\R^n}\|\p^j f\|\leq m, \hb{ for }|j|\leq [\kappa], \hb{ and}\\
 &\|\p^jf(x)-\p^jf(y)\|\leq m\|x-y\|^{\kappa-[\kappa]} \hb{ for } |j|=[\kappa],\  x,\ y\in\R^n\big\}.
\end{split}	
\end{equation}

 Let $\Omega\subset\R^n$ be a bounded domain with smooth boundary.  We refer to  \cite[section 4.1]{GiTr} for the standard definition of the Hölder spaces ${\cal C}^{k,\nu}(\bar\Omega)$ and we note that the definition extends in an obvious way to $T$-valued maps. We will denote by ${\cal C}^{k,\nu}(\bar\Omega,T)$ or ${\cal C}^{\kappa}(\bar\Omega,T)$ the resulting Banach space.  Using the extension Lemma  \cite[Lemma 6.37]{GiTr} we   obtain an equivalent definition of the space ${\cal C}^{\kappa}(\bar\Omega,T)$:
 \begin{equation}
 {\cal C}^{\kappa}(\bar\Omega,T)=\{f\in {\cal C}^0(\bar\Omega,T)|\ \exists \tilde f\in \mathrm{Lip}^\kappa(\R^n,T) \hbox{ such that }\tilde f|_{\bar\Omega}=f\}.
 \end{equation}
 This shows that the restriction epimorphism
 $$
 |_{\bar \Omega}:\mathrm{Lip}^\kappa(\R^n,T)\to {\cal C}^{\kappa}(\bar\Omega,T)
 $$
 induces an isomorphism of Banach spaces
 $$
 \qmod{\mathrm{Lip}^\kappa(\R^n,T)}{\{f\in \mathrm{Lip}^\kappa(\R^n,T)|\ f|_{\bar\Omega}\equiv 0\}}\textmap{\simeq}{\cal C}^{\kappa}(\bar\Omega,T).
 $$
 We also define the Lipschitz space $\mathrm{Lip}^\kappa(\R^n_+,T)$ of $T$-valued maps on the half-space $\R^n_+\edf \{x\in \R^n|\ x_n\geq 0\}$ by
$$
\mathrm{Lip}^\kappa(\R^n_+,T)\edf \{f\in {\cal C}^0(\R^n_+,T)|\ \exists \tilde f\in \mathrm{Lip}^\kappa(\R^n,T) \hbox{ such that }\tilde f|_{\R^n_+}=f\}.
$$
Endowed with the norm $\|\cdot \|_{\mathrm{Lip}^\kappa}$ induced by the obvious isomorphism
$$
\qmod{\mathrm{Lip}^\kappa(\R^n,T)}{\{\psi\in \mathrm{Lip}^\kappa(\R^n,T)|\ \psi|_{\R^n_+}\equiv 0\} }\textmap{\simeq} \mathrm{Lip}^\kappa(\R^n_+,T),
$$
$\mathrm{Lip}^\kappa(\R^n_+,T)$ becomes a Banach space.

Let now $M$ ($\bar U$) be  an $n$-dimensional differentiable manifold (with boundary),  and  ${\cal A}_M$ (${\cal A}_{\bar U})$ be the maximal atlas (the set of charts) of $M$ ($\bar U$).  We define the spaces
\begin{equation}\label{C-kappa-M}
\begin{split}
{\cal C}^\kappa(M,T)\edf \{f\in &\ {\cal C}^0(M,T)|\ (\chi f|_{V_h})\circ h^{-1} \in\mathrm{Lip}^\kappa(\R^n,T) \hb{ for any } \   \\
&(M\stackrel{\rm open}{\supset} V_h\stackrel{h}\to W_h\stackrel{\rm open}{\subset} \R^n) \in{\cal A}_M	 \hb{ and  }\chi\in {\cal C}^\infty_c(V_h,\R)\},
\end{split}	
\end{equation}
\begin{equation}\label{C-kappa-bar-U}
\begin{split}
{\cal C}^\kappa(\bar U,T)\edf \{f \in &\ {\cal C}^0(\bar U,T)|\  (\chi f|_{V_h})\circ h^{-1} \in\mathrm{Lip}^\kappa(\R^n_+,T) \hb{ for any } \   \\
&(\bar U\stackrel{\rm open}{\supset} V_h\stackrel{h}\to W_h\stackrel{\rm open}{\subset} \R^n_+) \in{\cal A}_{\bar U}	 \hb{ and  }\chi\in {\cal C}^\infty_c(V_h,\R)\}.
\end{split}	
\end{equation}

The space  ${\cal C}^\kappa(M,T)$ (${\cal C}^\kappa(\bar U,T)$) is naturally a Fréchet space.  When $M$ ($\bar U$) is compact, the topology of this space can be defined by a single norm, so it becomes a Banach space. More precisely, for a finite atlas ${\cal A}\subset {\cal A}_M$ (${\cal A}\subset {\cal A}_{\bar U}$) of a compact manifold $M$ (with boundary $\bar U$) and a partition of unity $(\chi_h)_{h\in {\cal A}}$ subordinate to the open cover $(V_h)_{h\in {\cal A}}$ of $M$ ($\bar U$), we obtain a norm on  ${\cal C}^\kappa(M,T)$ (${\cal C}^\kappa(\bar U,T)$), defining its topology, given by
\begin{equation}\label{Norm-On-Ckappa}
\| f\|_{{\cal C}^\kappa}=\sum_{h\in {\cal A}} \|(\chi_h f|_{V_h})\circ h^{-1}\|_{\mathrm{Lip}^\kappa},
\end{equation}
where $\|\cdot \|_{\mathrm{Lip}^\kappa}$ stands for the norm defined above on the space  $\mathrm{Lip}^\kappa(\R^n,T)$ (respectively $\mathrm{Lip}^\kappa(\R^n_+,T)$). In particular we obtain a third equivalent definition of the Banach space ${\cal C}^{\kappa}(\bar\Omega,T)$  associated with a bounded domain $\Omega\subset\R^n$ with smooth boundary.
\\

Let   $\Omega$ be such a domain. A   $T$-valued differential form of degree $d$ on $\Omega$ can be regarded as a map $\Omega\to L^d_{\rm alt}(\R^n,T)$ with values in the space $L^d_{\rm alt}(\R^n,T)$ of $T$-valued alternating $d$-linear maps on $(\R^n)^d$. Using the identification
$$
{\cal C}^k(\Omega,\extp^d_{\;\Omega}\otimes T)\textmap{\simeq}{\cal C}^k(\Omega,L^d_{\rm alt}(\R^n,T)), 
$$
we obtain a natural definition of the Hölder space ${\cal C}^\kappa(\bar \Omega,\extp^d_{\;\bar\Omega}\otimes T)$: one just replaces $T$ by $L^d_{\rm alt}(\R^n,T)$ in the definition of ${\cal C}^\kappa(\bar \Omega,T)$. More generally, using formulae similar to (\ref{C-kappa-M}), (\ref{C-kappa-bar-U})  we obtain -- for a differentiable manifold (with boundary) $M$  ($\bar U$) --  the Fréchet spaces ${\cal C}^k(M,\extp^d_{\;M}\otimes T)$ (respectively ${\cal C}^k(\bar U,\extp^d_{\;\bar U}\otimes T)$); these spaces become Banach spaces when $M$ ($\bar U$) is compact. In this case, choosing a pair $({\cal A},(\chi_h)_{h\in {\cal A}})$ consisting of a finite atlas of $M$ ($\bar U$) and a partition of unity subordinate to  the open cover $(V_h)_{h\in {\cal A}}$, we obtain -- using a formula similar to (\ref{Norm-On-Ckappa}) -- defining norms on the spaces ${\cal C}^\kappa(\bar \Omega,\extp^d_{\;\bar\Omega}\otimes T)$, $d\geq 0$.

 \subsection{The Newlander-Nirenberg theorem for vector bundles}\label{DolbeaultSect}
Let $U$ be a  connected $n$-dimensional complex manifold and $E$  a differentiable complex vector bundle of rank $r$ on $U$. Let 
$$\delta: A^0(U,E)\to A^{0,1}(U,E)$$
be a Dolbeault operator (semi-connection)  on $E$, i.e. a first order differential operator  satisfying the Leibniz rule $\delta(f \sigma)=\bp f \sigma+f\delta\sigma$ (see for instance \cite[section 2.2.2]{DK} \cite[section 1]{LO}, \cite[section 4.3]{LT}). We  denote by the same symbol the natural extension $A^{0,q}(U,E)\to A^{0,q+1}(U,E)$ and recall that $\delta^2:A^0(U,E)\to A^{0,2}(U,E)$ is an order 0 operator, so it is given by a form $F_\delta\in A^{0,2}(U,\End(E))$. With respect to a local trivialization,  $\delta$  has the form $\bp+\alpha$ for a  $\gl(r,\C)$-valued form $\alpha$ of type (0,1), and then, in the same trivialization, $F_\delta$ is given by $\bp\alpha+\alpha\wedge\alpha$.
 
By the vector bundle version of the Newlander-Nirenberg theorem, the $\End(E)$-valued (0,2)-form $F_\delta$ is the obstruction to the integrability of $\delta$  \cite[Proposition p. 419]{Gr}, \cite[Theorem 5.1]{AHS}, \cite[Proposition I.3.7]{Ko}, \cite[Theorem 2.1.53]{DK}):

\begin{thr} [The Newlander-Nirenberg theorem for vector bundles]  
 Let $\delta$ be a Dolbeault operator on $E$. The following conditions are equivalent:
 \begin{enumerate}
 	\item $F_\delta=0$.
 	\item $\delta$ is integrable in the following sense: for any point $x\in U$ there exists an open neighborhood $W$ of $x$ and a  frame $(\theta_1,\dots,\theta_r)\in A^0(W,E)^r$ with $\delta \theta_i=0$.
 \end{enumerate}
 If this is the case, $\delta$ defines a holomorphic structure $\hg_\delta$ on $E$. For an open set $W\subset U$, a section $\sigma\in A^0(W,E)$ is holomorphic with respect to $\hg_\delta$ if and only if $\delta\sigma=0$. 
 \end{thr}

The map $\delta\mapsto \hg_\delta$ defines a bijection between the set of integrable Dolbeault operators and the set of holomorphic structures on $E$. This  result has   important consequences:  the set of isomorphism classes of holomorphic bundles which are differentiably isomorphic to $E$ can be identified with the quotient ${\cal D}^{\rm int}_E/\Aut(E)$  of gauge classes of integrable Dolbeault operators on $E$. Therefore ideas and techniques from gauge theory can be used in the construction of moduli spaces of holomorphic bundles. This idea has been used in    \cite{LO}  to give a gauge theoretical construction of the moduli space of simple holomorphic bundles  with fixed differentiable type. 
\vspace{2mm}

 \subsection{Vector fields on principal bundles}\label{VectorFieldsOnP}
 
 Let $G$ be a Lie group, $P$ a differentiable manifold, and $P\times G\to P$ a smooth right action of $G$ on $P$.  The infinitesimal action of the Lie algebra $\g$ of $G$ on $P$  can  be regarded a $\g^*$-valued vector field $\nu\in \Gamma(P,T_P\otimes \g^*)$. Explicitly $\nu$ is given by
 $$
 \nu_y(a)=a^\#_y \ \forall y\in P,\, \forall a\in\g.
 $$
For any map  $\lambda\in {\cal C}^\infty(P,\g)$ we obtain a vector field $\nu\cdot\lambda$ given by 
$$(\nu\cdot\lambda)_y=\nu_y(\lambda(y))=\lambda(y)^\#_y.$$
In other words $\nu\cdot\lambda$ is the image of the $\g^*\otimes\g$-valued vector field $\nu\otimes\lambda$ under the canonical vector bundle morphism $T_P\otimes(\g^*\otimes\g)\to T_P$. We will use the simpler notation $\lambda^\nu$ for the vector field $\nu\cdot \lambda$. If $\lambda$ is the constant map associated with $a\in\g$, then $\lambda^\nu=a^\#$.
 
 The $\g^*$-valued vector field $\nu$ is obviously invariant under any local diffeomorphisms $P\stackrel{\rm open}{\supset} U\textmap{f}V\stackrel{\rm open}{\subset} P$ which commutes with the infinitesimal $G$-action on $P$, i.e. such that $f_*(a^\#|_U)=a^\#|_V$.
Using this fact we obtain: 
 \begin{re}
 Let $\xi\in {\cal X}(P)$ be vector field whose associated local 1-parameter group of diffeomorphisms $(\varphi_t)_t$ commutes with the infinitesimal $G$-action on $P$, i.e. it satisfies  the property
 \vspace{2mm}\\
{\bf P:} For any $x\in P$ there exists $\varepsilon_x>0$ and an open neighborhood $U_x$  of $x$ in $P$ such that for any $t\in (-\varepsilon_x,\varepsilon_x)$, the local diffeomorphism $\varphi_t$ is defined on $U_x$ and $\varphi_t:U_x\to \varphi_t(U_x)$  commutes with the infinitesimal $G$-action on $P$ in the sense defined above. 
  \vspace{2mm}\\ 
Then $L_\xi(\nu)=0$.
 \end{re}
 \begin{proof}
 Indeed, property {\bf P} implies  $\varphi_{t*}(\nu)_x=\nu_x$ for any $t\in (-\varepsilon_x,\varepsilon_x)$  for which $\varphi_t^{-1}(x)\in U_x$. Differentiating with  respect to $t$ at $t=0$ we obtain $L_\xi(\nu)_x=0$.
 \end{proof}

Using \cite[Corollary 1.8 p. 14]{KN} applied to the vector fields $a^\#$, $a\in\g$ and  \cite[Corollary 1.11 p. 16]{KN}  it follows that $\xi$ has  property {\bf P} if only if $[\xi,a^\#]=0$ for any $a\in \g$, i.e. if only if $\xi$ is invariant under the infinitesimal $G$-action on $P$.

  We obtain:
 \begin{re}\label{xi-nu-lambda-rem}
 Let $\xi\in {\cal X}(P)$ be a  vector field on $P$ which is invariant under the infinitesimal $G$-action on $P$, i.e. such that $[\xi,a^\#]=0$  for any $a\in \g$. Then 	$L_\xi(\nu)=0$, in particular for any map $\lambda\in {\cal C}^\infty(P,\g)$ we have $L_\xi(\lambda^\nu)= L_\xi(\lambda)^\nu $, i.e.
 \begin{equation}\label{xi-nu-lambda}
 [\xi,  \lambda^\nu]= \xi(\lambda)^\nu.
 \end{equation}
 \end{re}
 
 Suppose now that $p:P\to U$ is a principal $G$-bundle. For any $\lambda\in {\cal C}^\infty(P,\g)$ the vector field $ \lambda^\nu$ is vertical. If $\lambda$ is $\Ad$-equivariant, i.e. if $\lambda$ belongs to ${\cal C}^\infty_{\Ad}(P,\g)=A^0(U, \Ad(P) )$,   then the vector field $ \lambda^\nu$ is $G$-invariant, so Remark \ref{xi-nu-lambda-rem} applies and (\ref{xi-nu-lambda}) gives
 \begin{equation}\label{xi-nu-lambda-lambda'}
 [ \lambda^\nu,  \lambda'^\nu]= ((\lambda^\nu)(\lambda'))^\nu.
 \end{equation}
 for any $\lambda'\in {\cal C}^\infty(P,\g)$. If $\lambda'$ is also $\Ad$-equivariant, we have
 \begin{equation}
(\lambda^\nu)(\lambda')=-[\lambda,\lambda'].	
 \end{equation}
 This follows by noting that, via the diffeomorphism $f_y:G\to y_0G$ associated with a point $y_0\in P$, the restriction of $\lambda'$ to the fiber $y_0G$ is given by $g\mapsto \Ad_{g^{-1}}(\lambda'(y_0))$, whereas the restriction of $\lambda^\nu$ to $y_0G$ is the right invariant vector field associated with $\lambda(y_0)$.
 Therefore, we obtain
 \begin{re}\label{nu-cdot-lambda-nu-cdot-lambda'}
 Let $p:P\to U$ be a principal $G$-bundle, and $\lambda$, $\lambda'\in {\cal C}^\infty_{\Ad}(P,\g)=A^0(U,\Ad(P))$. Then 
 $$
 [\lambda^\nu,  \lambda'^\nu]=-[\lambda,\lambda']^\nu.
 $$
 \end{re}

This formula can also be obtained by noting that, for $\lambda\in {\cal C}^\infty_{\Ad}(P,\g)$, $\lambda^\nu$ is the vector field (infinitesimal transformation) associated with $\lambda$ regarded as element in the Lie algebra $\Gamma(U,\Ad(P))$ of the gauge group $\Aut(P)=\Gamma(U,\iota(P))$. Since $\Aut(P)$ acts on $P$ from the left, the linear map ${\cal C}^\infty_{\Ad}(P,\g)=\mathrm{Lie}(\Aut(P))\to {\cal X}(P)$ is an  anti-homomorphism of Lie algebras.

%$$
%g_t^{-1} \lambda'(y_0) g_t,  \ g^{-1}e^{-\lambda(y_0)t} \lambda'(y_0) e^{\lambda(y_0)t}g
%$$

\end{document}